\documentclass[a4paper]{article}
\usepackage{setspace} 
\usepackage{srcltx}

\usepackage{epsfig}

\usepackage{color}

\usepackage{latexsym}

\usepackage{amsmath}

\usepackage{amsthm}
\usepackage{amsfonts}
\usepackage{siunitx}

\usepackage{hyperref}

\newtheorem{thm}{Theorem}[section]

\newtheorem{lem}[thm]{Lemma}

\newtheorem{defn}[thm]{Definition}
\theoremstyle{remark}
\newtheorem{remark}[thm]{Remark}

\title{{\bfseries Overlapping Schwarz Methods with Adaptive Coarse Spaces for Multiscale Problems in 3D.}}

\author{Erik Eikeland\thanks{Department of Computing, Mathematics and Physics, Bergen University College, 5020 Bergen, Norway ({Erik.Eikeland@hib.no})}
\and Leszek Marcinkowski\thanks{Department of Mathematics, University of Warsaw, Banacha 2, 02-097 Warszawa, Poland ({Leszek.Marcinkowski@mimuw.edu.pl})}
\and Talal Rahman\thanks{Department of Computing, Mathematics and Physics, Bergen University College, 5020 Bergen, Norway ({Talal.Rahman@hib.no})} 
}

\begin{document}

\date{}
\maketitle
\begin{abstract}
We propose two variants of the overlapping additive Schwarz method for the finite element discretization of the elliptic problem in 3D with highly heterogeneous coefficients. The methods are efficient and simple to construct 
using the abstract framework of the additive Schwarz method, and an idea of adaptive coarse spaces.  In  one variant, the coarse space consists of finite element functions associated with the wire basket nodes and functions based on solving some generalized eigenvalue problem on the faces, and in the other variant, it contains  functions associated with the vertex nodes  with functions based on solving some generalized eigenvalue problems on subdomain faces and on subdomain edges. The functions that are used to build the coarse spaces  are chosen adaptively, they correspond to the eigenvalues that are smaller than a given threshold. The convergence rate of the preconditioned conjugate gradients method in both cases, is shown to be independent of the variations in the coefficients for sufficient number of eigenfunctions in the coarse space. Numerical results are given to support the theory.
\end{abstract}

\markboth{ERIK EIKELAND, LESZEK MARCINKOWSKI, AND TALAL RAHMAN}{Overlapping Schwarz Methods with Adaptive Coarse Spaces for Multiscale Problems in $3$D}

\section{Introduction}
Additive  Schwarz methods are among the most popular domain decomposition methods for solving partial differential equations, they are simple in construction and are inherently parallel, cf. \cite{dolean:2015:IDD,Mathew:2008:DDM,smith2004domain,toselli2005domain}. We consider overlapping additive Schwarz preconditioners for solving scalar elliptic equations with highly varying coefficient in the space $\mathbb{R}^3$. Based on the idea of adaptively constructing the coarse space through solving generalized eigenvalue problems in the lower dimensions, we propose two variants of the algorithm in three dimensions. The resulting system has a condition number which is inversely proportional to $\lambda^{*}$, where $\lambda^{*}$ is the threshold for the good eigenvalues and enters as a parameter for the coarse space in the algorithm. 
The term lower dimensions refers to the fact that the eigenvalue problems are solved either on subdomain faces or on edges. The methods are effective, inherently parallel, and simple to construct. Additive Schwarz method with adaptive coarse space has been considered recently using a different idea which is based on solving generalized eigenvalue problem in the overlap, cf. \cite{dolean:2015:IDD,Dolean:2012:ATL,Nataf:2011:CSC}. 

To motivate the use of solving eigenvalue problems on subdomain boundaries we take a brief look at the steps of estimating convergence of two level overlapping Schwarz methods in its abstract framework, cf. \cite{Mathew:2008:DDM,smith2004domain,toselli2005domain}. As we look for the solution in the finite element space, the analysis of the method in the additive Schwarz framework is based on the assumption that any function $u$ in the space can be split into its coarse component $u_0$ and local components associated with the subdomains, and the splitting is stable with respect to the energy norm. The condition number bound of the preconditioned system then depends explicitly on the constant $C_0^2$ entering into the stability estimate of the splitting, cf. \cite{Mathew:2008:DDM,smith2004domain,toselli2005domain}. Consequently, how robust the method will be with respect to mesh parameters and varying coefficients, depends on how robust $C_0^2$ is with respect to them. Even with highly varying coefficients, as long as the variations are strictly inside subdomains, it is possible to see that this constant will not depend on the variation of the coefficients, cf. e.g. \cite{graham2007domain}. 
The difficulty arises only when we deal with coefficients which vary along subdomain boundaries. 
In which case, the crucial term which needs to be bounded for the stability estimate is the boundary term $\sum_i C/h^2\|\sqrt{\alpha}(u-u_0)\|^2_{L^2(\Omega^h_i)}$, $\Omega^h_i$ being the layer of all elements along the subdomain boundary $\partial \Omega_i$, $\alpha$ the varying coefficient, and $C$ a constant.  
Because the $\alpha$ is varying, estimating the $L^2$ term say using the Poincar\'e or a weighted Poincar\'e inequality will necessarily introduce the contrast into the estimate, i.e. the ratio between the largest and the smallest values of $\alpha$, unless strong assumptions on the distribution of the coefficient are made, cf. \cite{Pechstein:2013:WPI}. A standard multiscale coarse space alone, cf. \cite{graham2007domain,Pechstein2013:FBT}, cannot make the method robust with respect to the contrast, unless some form of enrichment of the coarse space is used that can capture the strong variations along the subdomain boundary and improves the approximation. 
Including selected eigenfunctions of properly defined eigenvalue problems over the subdomain boundary into coarse spaces, enable us  to capture those variations and provide estimates with constants independent of the contrast, which is the main motivation behind the methods presented here. 

The idea of adaptively constructing coarse spaces using eigenfunctions of certain eigenvalue problems has already attracted a lot of interest in the recent years. Some earlier works in this direction are found in \cite{Bjorstad-Koster-Krzyzanowski:HPC:2002,Bjorstad-Krzyzanowski:PPAM2001} around the Neumann-Neumann type substructuring domain decomposition method, and \cite{Chartier:2003:SAMGe} around the algebraic multigrid method, although they were not addressed to solve multiscale problems. In the context of multiscale problems this idea has only recently started to emerge, with its first appearance in \cite{Galvis:2010:DDM,Galvis:2010:DDMR}, as well as in \cite{Nataf:2011:CSC}, and later on in \cite{chung2016adaptive,Dolean:2012:ATL,Efendiev:2012:RTD,Efendiev:2012:RDD,atle2015harmonic,Spilane:2014:ARC} in the additive Schwarz framework. This idea of adaptively constructing coarse space, in other word the primal constraints, for the FETI-DP and BDDC substructuring domain decomposition methods has also been developed and extensively analyzed, cf.
\cite{kim2015bddc,KRR:2015:FDMACS,klawonn2016comparison,mandel2007adaptive,sousedik2013adaptive} in 2D and \cite{calvo2015adaptive,kim2016bddc,klawonn:2016:adaptive,oh:2016:BDDC,pechstein2016unified} in 3D.
 
The rest of this paper is organized as follows: In Section~\ref{sec:differential_problem} we define our problem and its discrete formulation. 
In Section~\ref{sec:two_level_additive_schwarz_preconditioner} we describe the overlapping Schwarz preconditioner, in Section~\ref{sec:coarse_space} we introduce the two new coarse spaces, and in Section~\ref{sec:convergence} we establish the convergence estimate for the preconditioned system. Finally, in Section~\ref{sec:numerical_results}, we present some numerical results of our method. 

\section{Differential problem and discrete formulation}\label{sec:differential_problem}
Here we present our continuous test problem, the scalar elliptic equation with coefficients piecewise constant on each fine scale element and its discrete representation.
Find $u \in H^1_0(\Omega)$ such that
\begin{equation}\label{problem}
 a(u,v)=f(v),\quad v\in H^1_0(\Omega),
\end{equation}
where
\begin{equation}\label{eqn:bilinear}
 a(u,v):=(\alpha(\cdot)\nabla u,\nabla v)_{L^2(\Omega)} \quad \mbox{and} \quad f(v):=\int_\Omega fv dx.
\end{equation}
We assume that $\alpha \in L^\infty(\Omega)$, $\alpha(x)\geq \alpha_0 > 0$, and $f \in L^2(\Omega)$ with $\Omega$ being a polyhedral region in the space $\mathbb{R}^3$.

Let $\mathcal{T}_h(\Omega)$ be a quasi uniform triangulation of $\Omega$ into fine tetrahedral elements $\tau$, where $h=\max_{\tau\in\mathcal{T}_h(\Omega)}$diam$(\tau)$ is the parameter of
$\mathcal{T}_h(\Omega)$, cf. e.g. \cite{Brenner:2008:MTF}. Keeping in mind that a tetrahedral element consists of four triangular faces and six edges, we denote the fine element face by $\tau_t$ and the finite element edge by $\tau_e$.
Let $V_h=V_h^0(\Omega)$ be the finite element solution space of piecewise linear continuous functions:
\begin{eqnarray*}
   V^h=V^0_h(\Omega):=\left\{v\in C_0(\Omega): \; v_{|\tau}\in P_1(x), \quad v_{|\partial\Omega}=0\right\},
\end{eqnarray*}
where $v_{|\tau}$ is the function restricted to an element $\tau\in\mathcal{T}^h(\Omega)$ and $P_1(x)$ is the set of linear polynomials. Note that as the gradient of
$u\in V^h$ is piecewise constant thus $(\rho \nabla u \nabla v)_{L^2(\tau)}=\nabla u \nabla v \int_\tau \rho \:dx$.  
Thus without any loss of generality we  assume that 
$\rho(x)=\alpha_\tau$  for $x\in\tau$ where $\alpha_\tau$ is a positive constant.

The discrete problem is then defined as: Find $u_h \in V^0_h(\Omega)$ such that
\begin{equation}\label{eqn:discrete_problem}
 a(u_h,v)=f(v),\quad v\in V^0_h(\Omega).
\end{equation}
This problem has a unique solution by the Lax-Milgram theorem. 
By using standard nodal basis functions $\phi_i$ with $i\in\{1,...,N_\eta\}$, where $N_\eta$ are the number of individual vertex nodes of the tetrahedrons in $\Omega$,
the above equation may be stated as a system of algebraic equations
\begin{equation}\label{originaldiscreteproblem}
 Au_h=f_h,
\end{equation}
where $(A)_{i,j}=a(\phi_j,\phi_i)$, ${(f_{h})}_i=f(\phi_i)$ and ${(u_h)}_j=u_h(x_{j})$. Here $x_j$ is the coordinate value of the node $j$. The resulting symmetric system is in general very ill-conditioned; any standard iterative method, like the Conjugate Gradients, e.g. cf \cite{Malek:2015:PCG}, may preform badly due to the
ill-conditioning of the system. 
The aim is to introduce an additive Schwarz preconditioner for the original problem (\ref{eqn:discrete_problem}) in order to obtain a well-conditioned system for which the convergence of the conjugate gradient method will be independent of any variations in the coefficient, thereby improving the overall performance. 

\section{Additive Schwarz method}\label{sec:two_level_additive_schwarz_preconditioner}
The two level additive Schwarz method is well known and well understood in the literature, we refer to \cite[Chapter 3]{toselli2005domain} for an overview of the method. 

\subsection{Geometric structures}\label{subsec:geometic_structures}
Let $\Omega$ be partitioned into a set of $N$ nonoverlapping subdomains (or generalized subdomains), $\{\Omega_i\}_{i\in I}$, $I=\{1,\ldots,N\}$, such that each $\overline{\Omega}_i$ (the closure of ${\Omega}_i$) is a sum of elements (fine elements) from $\mathcal{T}_h$, ${\Omega}_i\cap{\Omega}_j = \emptyset$ for $i\neq j$, and $\overline{\Omega} = \cup_{i\in I}\overline{\Omega}_i$. The intersection between two closed subdomains is either an empty set, a closed face (or a closed generalized face which is a sum of closures of fine element faces), a closed edge (or a closed generalized edge which is a sum of closures of fine element edges), or a vertex.
Thus, the triangulation $\mathcal{T}^h(\Omega)$ is aligned with the subdomains $\Omega_i$.  Each subdomain $\Omega_i$ inherits its own local triangulation $\mathcal{T}^h(\Omega_i)=\{\tau\in\mathcal{T}^h(\Omega):\tau\subset\overline{\Omega}_i\}$ such that $\bigcup_i\mathcal{T}^h(\Omega_i)=\mathcal{T}^h(\Omega)$. 
The corresponding set of overlapping subdomains $\{\Omega_i'\}_{i\in I}$ is then defined as follows: extend each subdomain $\Omega_i$ to $\Omega_i'$, by adding to $\Omega_i$ a layer of elements, i.e. sum of $\overline{\tau}_k\in\mathcal{T}^h(\Omega)$ such that $\overline{\tau}_k\cap\partial\Omega_{i}\neq\emptyset$. 

Since the subdomains inherit their triangulation from the global triangulation $\mathcal{T}^h(\Omega)$, the nodes, the edges, and the faces of the tetrahedral elements along the interface $\Gamma=\bigcup_i\partial\Omega_i\setminus\partial\Omega$ will be matching across. The interface is composed of three basic structures: open (generalized) faces, open (generalized) edges, and subdomain vertices. The set of all open edges and subdomain vertices constitute the structure which we call the wire basket, and we denote it by $\mathcal{W}$.
The set of all faces is given by $\mathcal{F}=\{\mathcal{F}_{kl} : \ \mathcal{\overline{F}}_{kl}=\partial\Omega_k\cap\partial\Omega_l, \ |\mathcal{F}_{kl}|>0, \ k\in I, \ l\in I, \ \mbox{and} \ k>l\}$, where $|\cdot|$ is a
natural measure of a surface. Note that since the subdomain vertices are matching across the interface, we have $\Gamma=\bigcup_{kl} \mathcal{\overline{F}}_{kl}\setminus\partial\Omega$.
A closed edge is the intersection between two closed faces, consisting of closed fine element edges. 
The set of all edges are denoted by $\mathcal{E}$.
And the set of subdomain vertices are denoted by $\mathcal{V}$.
The wire basket is defined as
$\mathcal{W}=\bigcup_k\mathcal{\overline{E}}_k\setminus\partial\Omega$.

In the same way as each subdomain inherits a 3D triangulation, each face $\mathcal{F}_{k l}$ inherits a 2D triangulation which we denote by $\mathcal{T}_h(\mathcal{F}_{k l})$, and each edge $\mathcal{E}_k$ inherits a 1D triangulation which we denote by $\mathcal{T}_h(\mathcal{E}_k)$.
For each of the structures, $\Omega$, $\overline{\Omega}$, $\Omega_k$, $\overline{\Omega}_k$, $\mathcal{F}$, $\mathcal{E}$, and $\mathcal{W}$, we use $\Omega_h$, $\overline{\Omega}_h$, $\Omega_{k,h}$, $\overline{\Omega}_{k,h}$, $\mathcal{F}_h$, $\mathcal{E}_h$, and $\mathcal{W}_h$, respectively, to denote the corresponding set of nodal points (vertices of the elements of $\mathcal{T}_h$) which are on the structure.


\subsection{Space decomposition, subproblems, and preconditioned system}
 Let the two local subspaces on $\Omega_k$, be defined as
\begin{eqnarray}
 V_h(\Omega_k)&:=&\{u_{|\overline{\Omega}_k}: u \in V_h\}. \nonumber\\
 V^0_h(\Omega_k)&:=&V_h(\Omega_k)\cap H^1_0(\Omega_k). \label{eq:local_space_zero}
\end{eqnarray}
Functions of $V^0_h(\Omega_k)$ are extended by zero to the rest of the domain $\overline\Omega\setminus\overline\Omega_k$. For the ease of representation, we denote this extended space by the same symbol, that is using $V^0_h(\Omega_k)$.
We can repeat these definitions for extended domains $\Omega_i'$.
Let us decompose the finite element solution space into a coarse space and $N$ local subspaces
\begin{equation}
 V^h=V_0^{(k)}+\sum_{i=1}^N V_i
\end{equation}
Here $V_i=V^0_h(\Omega_i')$ are the local function spaces associated with the overlapping subdomains $\Omega_i'$ for $i=1, \ldots, N$ extended by zero to the rest of $\Omega$, cf. (\ref{eq:local_space_zero}).
Further, $V_0^{(k)},k=1,2$ is the coarse space. In the next section, we introduce the two coarse spaces, the wire basket based coarse space $V_0^{(1)}=V_{\mathcal W}$, cf. (\ref{eq:wire-coarse-sp}),  and the vertex based coarse space $V_0^{(2)}=V_{\mathcal V}$, cf. (\ref{eq:vertex-coarse-sp}).
In both cases these are relatively much smaller subspaces of the finite element space $V_h$.

We define projection like operators, $P_0^{(k)}$ for $k=1, 2$, and $P_i$ for $i=1,\ldots,N$, as follows,
\begin{eqnarray*}
a(P_0^{(k)}u,v)=a(u,v)\quad \forall v\in V_0^{(k)} \ k=1,2\\
a(P_iu,v)=a(u,v)\quad \forall v\in V_i, i=1,\ldots,N.
\end{eqnarray*}


Defining the additive operator $P^{(k)}=P_0^{(k)} + \sum_i P_i$, we get the following system of equation equivalent to the original problem  (\ref{eqn:discrete_problem}) 
in the operator form,
\begin{equation}\label{eq:goodsys}
 P^{(k)}u=g.
\end{equation}

\section{Coarse spaces with enrichment}\label{sec:coarse_space}
For our the additive Schwarz method we introduce two alternative coarse spaces. The first one is based on enriching a wire basket coarse space, and the second one is based on enriching a vertex based coarse space. 

Discrete harmonic extensions are used to extend our functions from subdomain boundaries into subdomains. 
We define our discrete harmonic extension operator below.
\begin{defn}\label{def:DHE}
Let $u|_{\partial \Omega_k}$ be the $u$ restricted to $\partial \Omega_k$. We define $\mathcal{H}_k:V_h(\Omega_k)\rightarrow V_h(\Omega_k)$ as the discrete harmonic extension operator in $\Omega_k$ as follows,
\begin{equation} 
\left\{
\begin{array}{lcr}
a_{|\Omega_k}(\mathcal{H}_k u,v)=0&\quad&\forall v\in V^0_h(\Omega_k),\\
 \mathcal{H}_ku=u_{|\partial\Omega_k} &\quad&       \mathrm{on} \quad \partial \Omega_k.
\end{array}\right.
\end{equation}
A function $u\in V_h(\Omega_k)$ is locally discrete harmonic in $\Omega_k$ if $\mathcal{H}_ku=u$. If
for any $u\in V_h$, we have that all its restrictions to the subdomains are locally discrete harmonic then this function is (piecewise) discrete harmonic in $\Omega$.
\end{defn}
Functions that are locally discrete harmonic have the minimum energy property locally. This property is well known, but for completeness we restate it in the following lemma, cf. \cite{toselli2005domain} for further details on discrete harmonic extensions.
\begin{lem} \label{MinEnHarm}
Let $u\in V_h(\Omega_k)$ be given such that $u=\mathcal{H}_k u$ in the sense of Definition \ref{def:DHE}, then it follows that
\begin{equation}
 a_{|\Omega_k}(u,u)=\min_{\{v|
   v=u \text{ on } \partial \Omega_k\}} a_{|\Omega_k}(v,v).
\end{equation}
\end{lem}

\subsection{Wire basket based coarse space}\label{sec:wire-coarse-sp}
The wire basket based coarse space consists of basis functions, one for each node in the wire basket $\mathcal{W}$, plus eigenfunctions corresponding to the first few eigenvalues of generalized eigenvalue problems associated with the faces, cf. Definition~\ref{defGeF}.

For any face $\mathcal{F}_{k l}$, let $V_h(\mathcal{F}_{k l})$ be the space of piecewise linear continuous functions on the face $\mathcal{F}_{k l}\in\mathcal{F}$, and $V_h^0(\mathcal{F}_{k l})$ the corresponding subspace of functions with zero boundary values, that is $$V_h^0(\mathcal{F}_{k l}) := \{v\in V_h(\mathcal{F}_{k l}): \ v(x)=0, \ \forall x\in \partial \mathcal{F}_{k l}\}.$$

We need the following interpolation operator $I_\mathcal{W}:V_h\rightarrow V_h$.
\begin{defn} \label{def:WireB-interpolant}
For any $u\in V_h$, let $I_\mathcal{W}u$ be discrete harmonic (cf. Definition \ref{def:DHE}) such that
\begin{eqnarray*}
 (I_{\mathcal{W}} u) (x) &=& u(x) \quad \forall x\in \mathcal{W}_h, \\
 a_{\mathcal{F}_{k l}}((I_{\mathcal{W}}u)_{|\mathcal{F}_{k l}},v)&=&0 \qquad \forall v\in V_h^0(\mathcal{F}_{k l}) \quad\forall \mathcal{F}_{k l}\subset \Gamma,
\end{eqnarray*}
where
\begin{eqnarray}
a_{\mathcal{F}_{k l}}(u,v)&:=&\sum_{\tau_t\in \mathcal{T}_h(\mathcal{F}_{k l})}\overline{\alpha}_{\tau_t}\int_{\tau_t }\nabla u \nabla v dx \quad u,v\in V_h(\mathcal{F}_{k l}) \label{eq:aFkl}
\end{eqnarray}
and $\overline{\alpha}_{\tau_t}=\max\{\alpha_{\tau_-},\alpha_{\tau_+}\}$ for a triangle $\tau_t$ which is an element of the 2D triangulation of the face $\mathcal{F}_{k l}$, such that $\overline{\tau}_t=\partial\tau_-\cap\partial \tau_+$
for $\tau_+\in \mathcal{T}_h(\Omega_k)$ and $\tau_-\in \mathcal{T}_h(\Omega_l)$.
\end{defn}

For each face $\mathcal{F}_{k l}$, we define the following generalized eigenvalue problem.
\begin{defn} \label{defGeF}
For $i=1,\ldots,\hat M$, where $\hat M=\dim(V_h^0(\mathcal{F}_{k l}))$, find $(\lambda^i_{\mathcal{F}_{k l}}, \xi^i_{\mathcal{F}_{k l}})\in\left(\mathbb{R}\times V^0_h(\mathcal{F}_{k l})\right)$ such that
$\lambda_{\mathcal{F}_{k l}}^1
\leq \ldots \leq \lambda_{\mathcal{F}_{k l}}^{\hat{M}}$ and 
\begin{equation}\label{eq:GeF}
 a_{\mathcal{F}_{k l}}(\xi^i_{\mathcal{F}_{k l}},v)=\lambda_{\mathcal{F}_{k l}}^i b_{\mathcal{F}_{k}}(\xi^i_{\mathcal{F}_{k l}}, v)\quad \forall v\in V^0_h(\mathcal{F}_{k l})
\end{equation}
where $a_{\mathcal{F}_{k l}}(\cdot,\cdot)$ is  defined in (\ref{eq:aFkl}), and $b_{\mathcal{F}_{k l}}(\cdot,\cdot)$ as follows,
\begin{equation} \label{eq:def-b-face}
  b_{\mathcal{F}_{k l}}(u,v):=\sum_{x\in\mathcal{F}_{k l,h}}\overline{\alpha}_xu(x)v(x) \quad u,v\in V_h(\mathcal{F}_{k l})
\end{equation}
where $\overline{\alpha}_x=\max\{\alpha_\tau: x\in \partial \tau, \tau \in \mathcal{T}_h\}$.
Here $V^0_h(\mathcal{F}_{k l})$ is the space of piecewise continuous functions on $\mathcal{T}_h(\mathcal{F}_{k l})$ which are equal to zero on 
$ \partial \mathcal{F}_{k l}$.
\end{defn}
 Note that the bilinear forms $a_{\mathcal{F}_{k l}}(\cdot,\cdot)$  and $b_{\mathcal{F}_{k l}}(\cdot,\cdot)$ in (\ref{eq:GeF}) are symmetric and positive definite on $V^0_h(\mathcal{F}_{k l})$.
We extend $\xi^i_{\mathcal{F}_{k l}}$ by zero to the whole interface $\Gamma$, and then as discrete harmonic function inside each subdomain, denoting the extended function by the same symbol.
Now, let
\begin{equation} \label{eq:WC-enrich}
V_{\mathcal{F}_{k l}}^{en}:=
\mathrm{span}
\{
\xi^i_{\mathcal{F}_{k l}}
\}_{i=1}^{m_{k l}}
\end{equation}
be the the space of eigenfunctions, where $m_{k l}$ is a nonnegative integer smaller than or equal to $\hat M$, a number which is either prescribed or decided adaptively.

The wire basket based coarse space is then defined as
\begin{equation} \label{eq:wire-coarse-sp}
  V_{\mathcal W}=I_{\mathcal{W}}V_h  + \sum_{\mathcal{F}_{k l} \in \mathcal{F}}  V_{\mathcal{F}_{k l}}^{en}.
\end{equation}
The space $I_{\mathcal{W}}V_h$ is a natural extension of the 2D multiscale coarse space introduced in \cite{graham2007domain} to 3D, and the face enrichment spaces $V_{\mathcal{F}_{k l}}^{en}$ are natural extensions of the edge enrichment spaces introduced in \cite{atle2015harmonic} to 3D.


\subsection{Vertex based  coarse space}\label{sec:vertex-coarse-sp}
The vertex based coarse space consists of basis functions, one for each node in $\mathcal{V}$, plus eigenfunctions corresponding to the first few eigenvalues of generalized eigenvalue problems associated with the edges, cf. Definition \ref{defGeE}, and the faces, cf. Definition \ref{defGeFI}.

For any edge $\mathcal{E}_i$, let  $V_h(\mathcal{E}_i)$ be the space of piecewise linear continuous functions on the edge $\mathcal{E}_i\in\mathcal{E}$, and $V_h^0(\mathcal{E}_i)$ the corresponding subspace of functions with zero boundary values, that is $$V_h^0(\mathcal{E}_i) := \{v\in V_h(\mathcal{E}_i): \ v(x)=0, \ \forall x\in \partial \mathcal{E}_i\}.$$

We need the following interpolation operator $I_{\mathcal V}:V_h\rightarrow V_h$.
\begin{defn}\label{def:vertex-interpolant}
For any $u\in V_h$, let $I_{\mathcal V}u$ be discrete harmonic (cf. Definition \ref{def:DHE}) such that
\begin{eqnarray*}
 (I_{\mathcal V} u) (x) &=& u(x) \quad \forall x\in \mathcal{V} \\
 (I_{\mathcal V} u) (x) &=& 0 \quad \forall x\in \Gamma_h\setminus \mathcal{W}_h \\
 a_{\mathcal{E}_i}((I_{\mathcal V}u)_{|\mathcal{E}_i},v) &=&0 \quad \forall v\in V_h^0(\mathcal{E}_i)
 \quad\forall \mathcal{E}_i\subset \mathcal{E}
\end{eqnarray*}
where
\begin{eqnarray}
a_{\mathcal{E}_i}(u,v)&:=&\sum_{e\in \mathcal{T}_h(\mathcal{E}_i)}\overline{\alpha}_e\int_e  u'  v' d x \quad u,v\in V_h(\mathcal{E}_i) \label{eq:aEi}
\end{eqnarray}
and $\overline{\alpha}_e=\max\{\tau\in \mathcal{T}_h: e\subset \partial \tau \}$.
\end{defn}

For each edge $\mathcal{E}_k$, we define the following generalized eigenvalue problem.
\begin{defn} \label{defGeE}
For $i=1,\ldots,M_{\mathcal E_k}$, where $M_{\mathcal E_k}=\dim(V_h(\mathcal{E}_k))$, find $(\lambda^i_{\mathcal{E}_k}, \xi^i_{\mathcal{E}_k})\in\left(\mathbb{R}\times V^0_h(\mathcal{E}_k)\right)$ such that
$\lambda_{\mathcal{E}_k}^1 \leq \ldots \leq \lambda_{\mathcal{E}_k}^{M_{\mathcal E_k}} $ and 
\begin{equation}\label{eq:GeE}
 a_{\mathcal{E}_k}(\xi^i_{\mathcal{E}_k},v)=\lambda_{\mathcal{E}_k}^i b_{\mathcal{E}_{k}}(\xi^i_{\mathcal{E}_k},v)\quad \forall v\in V^0_h(\mathcal{E}_k)
\end{equation}
where $a_{\mathcal{E}_k}(\cdot,\cdot)$ is  defined in (\ref{eq:aEi}), and 
\begin{eqnarray}
 \label{eq:def-b-edge}
  b_{\mathcal{E}_k}(u,v)&:=&h^{-1}\sum_{x\in\mathcal{E}_{k,h}}\overline{\alpha}_x u(x)v(x)  \quad u,v\in V_h(\mathcal{E}_k)
 \end{eqnarray}
with $\overline{\alpha}_x$ from Definition~\ref{defGeF}. 
\end{defn}
Note that, by definition, the bilinear forms $a_{\mathcal{E}_{k}}(\cdot,\cdot)$ and $b_{\mathcal{E}_k}(\cdot,\cdot)$
in (\ref{eq:GeE}) are both symmetric and positive definite on $V^0_h(\mathcal{E}_k)$.

For each face $\mathcal{F}_{k l}$, let 
$$
\overline{\mathcal{F}}_{k l}^B =\bigcup_{\{\tau_t
: \ \overline{\tau}_t\cap\partial\mathcal{F}_{k l}\neq\emptyset\}} \overline{\tau}_t
$$
be the sum of closed fine triangles 
on the face $\mathcal{F}_{k l}$ such that touch the wire basket, and $\mathcal{F}_{k l}^I=\mathcal{F}_{k l} \setminus
\overline{\mathcal{F}}_{k l}^B$
the interior of the sum of closed triangles that are lying strictly inside the face. Obviously,
$\overline{\mathcal{F}}_{k l} = \overline{\mathcal{F}}_{k l}^B \cup \mathcal{F}_{k l}^I$.


For each face $\mathcal{F}_{k l}$, we now define the following generalized eigenvalue problem.
\begin{defn} \label{defGeFI}
For $i=1,\ldots,\hat{M}$, where $\hat M 
=\dim(V_h^0(\mathcal{F}_{k l}))$, find $(\lambda^i_{\mathcal{F}_{k l,I}}, \xi^i_{\mathcal{F}_{k l,I}}) \in \left(\mathbb{R}\times V_h^0(\mathcal{F}_{k l})\right)$ such that
$\lambda_{\mathcal{F}_{k l, I}}^1\leq \ldots \leq \lambda_{\mathcal{F}_{k l, I}}^{\hat{M}}$ 
and
\begin{equation}\label{eq:GeFI}
 a_{\mathcal{F}_{k l,I}}(\xi^i_{\mathcal{F}_{k l,I}},v)=\lambda_{\mathcal{F}_{k l,I}}^i b_{\mathcal{F}_{k l}}(\xi^i_{\mathcal{F}_{k l,I}}, v)\quad \forall v\in V_h^0(\mathcal{F}_{k l})
\end{equation}
where
\begin{eqnarray} \label{eq:def-a-f_kl-I}
 a_{\mathcal{F}_{k l,I}}(u,v)&:=&\sum_{\tau_t\in\mathcal{F}_{k l}^I}\overline{\alpha}_{\tau_t}\int_{\tau_t}\nabla u \nabla v dx\quad u,v\in V_h(\mathcal{F}_{k l}),
\end{eqnarray}
where $\overline{\alpha}_{\tau_t}$ is  defined in Definition~\ref{def:WireB-interpolant}.
\end{defn}
Note that, by definition, the bilinear form $b_{\mathcal{F}_{k l}}(\cdot,\cdot)$ is symmetric and positive definite on $V_h^0(\mathcal{F}_{k l})$, while the bilinear form $a_{\mathcal{F}_{k l,I}}(\cdot,\cdot)$ in (\ref{eq:GeFI}) is symmetric and positive semidefinite on this space. We know its kernel, it is the one dimensional space containing functions that are constant over $\mathcal{F}_{k l}^I$, consequently, we have
$0=\lambda_{\mathcal{F}_{k l, I}}^1< \lambda_{\mathcal{F}_{k l, I}}^2 \leq  \ldots  \leq \lambda_{\mathcal{F}_{k l, I}}^{\hat{M}}$.

Analogously, as in the wire basket based coarse space, we extend $\xi^i_{\mathcal{E}_k}$ and $\xi^i_{\mathcal{F}_{k l,I}}$ by zero to the whole $\Gamma$, and then as discrete harmonic functions inside each subdomain, denoting the extended function by the same symbols.
Now, define the two spaces of eigenfunctions, one associated with the edges and one associated with the faces, as
\begin{equation} \label{eq:VC-enrich}
V_{\mathcal{F}_{k l,I}}^{en}:=
\mathrm{span}
\{
\xi^i_{\mathcal{F}_{k l,I}}
\}_{i=1}^{n_{k l}} \quad \mbox{and} \quad
V_{\mathcal E_k}^{en}:=
\mathrm{span}
\{
\xi^i_{\mathcal E_k}
\}_{i=1}^{m_k},
\end{equation}
respectively. Both $n_{k l}$ and $m_k$ are integers such that $1\leq n_{k l}\leq\hat M$ and $0 \leq m_k \leq M_{\mathcal E_k}$, and are either prescribed or decided adaptively.

The vertex based coarse space is then defined as
\begin{equation}\label{eq:vertex-coarse-sp}
  V_{\mathcal V}=I_{\mathcal{V}}V_h  + \sum_{\mathcal{F}_{k l} \in \mathcal{F}}  V_{\mathcal{F}_{k l,I}}^{en} +\sum_{\mathcal{E}_k \in \mathcal{E}}  V_{\mathcal E_k}^{en}.
\end{equation}

\begin{remark}
It should be pointed out here that many of the calculations indicated by the use of piecewise discrete harmonic extensions in the definitions above are redundant since, in practice, they will often be extensions of zero on subdomain boundaries. 
\end{remark}

\section{Convergence estimate for the preconditioned system}\label{sec:convergence}

In this section we prove that the condition number of our preconditioned system can be kept low, and independent of the contrast, if our coarse space enrichments are properly chosen. The main result is stated in Theorem~\ref{thm:ASM-cond-est}.

\begin{thm}\label{thm:ASM-cond-est}
Let $k=1$ and $k=2$ in the superscript refer to the two coarse spaces: the wire basket based coarse space and the vertex based coarse space, respectively. Then, for $P^{(k)}$, cf. (\ref{eq:goodsys}), we have
\begin{equation}
  \left(C_0^{(k)}\right)^{-2} a(u,u)\preceq a(P^{(k)}(u,u)\preceq a(u,u)
\end{equation}
with
\begin{eqnarray*}
  \left(C_0^{(1)}\right)^2 &=& 1+\max_{\mathcal{F}_{k l} \in \mathcal{F}} \frac{1}{\lambda_{\mathcal{F}_{k l}}^{m_{k l}+1}}
  \\
  \left(C_0^{(2)}\right)^2 &=& 1+
  \max\left\{\max_{\mathcal{F}_{k l} \in \mathcal{F}} (\lambda_{\mathcal{F}_{k l,I}}^{n_{k l}})^{-1},
        \max_{\mathcal{E}_k \in \mathcal{E}} (\lambda_{\mathcal{E}_k}^{m_k})^{-1}\right\}
\end{eqnarray*}
Here $\lambda_{\mathcal{F}_{k l}}^{i}$ is defined in (\ref{eq:GeF}),
$\lambda_{\mathcal{E}_k}^i$ is from (\ref{eq:GeE}), and $\lambda_{\mathcal{F}_{k l,I}}^i$ is from
(\ref{eq:GeFI}), and the integer parameters $m_{k l},n_{k l}, m_k$ are defined in (\ref{eq:WC-enrich})
or (\ref{eq:VC-enrich}), respectively.
\end{thm}
The proof is based on the abstract Schwarz framework, cf. e.g. \cite{smith2004domain,toselli2005domain,Mathew:2008:DDM}  and is given at the end of the section. The following lemmas are required for the proof.

\vspace{.2cm}
\begin{remark}
In case of constant $\alpha$ and regular mesh the minimal eigenvalue of the problems are of $\mathcal{O}((\frac{h}{H})^2)$ and the maximal eigenvalue is of $\mathcal{O}(1)$.
\end{remark}
\vspace{.2cm}

\begin{lem}\label{lem:L_Estimate}
Let $u\in V_h$ be discrete harmonic i.e. $u_k := u_{|\overline{\Omega}_k}=\mathcal{H}_k u$ in the a-norm $a_{|\Omega_k}(\cdot,\cdot)$ or be equal to zero at all interior nodes that are in $\Omega_{k,h}$, then it follows that
\begin{equation} \label{eq:lem:L_Estimate}
 a_{|\Omega_k}(u,u)\leq h \sum_{x \in \partial\Omega_{k,h}} \overline{\alpha}_x |u(x)|^2.
\end{equation}
In particular if $u$ is zero at $\mathcal{W}_h$ then
\begin{eqnarray}
a_{|\Omega_k}(u,u)&\preceq h\sum_{\mathcal{F}_{k l}\subset\partial\Omega_k} b_{\mathcal{F}_{k l}}(u,u),
\end{eqnarray}
and if $u$ is zero at all vertices of $\partial \Omega_k$ then
\begin{eqnarray}
a_{|\Omega_k}(u,u)&\preceq h^2\sum_{\mathcal{E}_i\subset\Omega_k} b_{\mathcal{E}_i}(u,u)+h\sum_{\mathcal{F}_{k l}\subset\partial\Omega_k} b_{\mathcal{F}_{k l}}(u,u).
\end{eqnarray}
\end{lem}
\begin{proof}
The first part of the proof follows from the fact that a discrete harmonic function has the minimal energy among all functions taking the same values on the boundary. Hence, $a_{|\Omega_k}(u,u)\leq a_{|\Omega_k}(\hat{u},\hat{u})$ for any $\hat{u}\in V_h(\Omega_k)$ which is equal to $u$ on $\partial\Omega_k$ and zero at the interior nodes $\Omega_{k,h}$. The other case means that $u=\hat{u}$. Consequently,
 \begin{eqnarray*}
  a_{|\Omega_k}(u,u)&\leq& a_{|\Omega_k}(\hat{u},\hat{u})
  = \sum_{\tau \in \mathcal{T}_h(\Omega_k)}
   \alpha_\tau \int_\tau  |\nabla \hat{u}|^2 \: d x
   \preceq h^{-2} \sum_{\tau \in \mathcal{T}_h(\Omega_k)}
   \alpha_\tau \int_\tau  |\hat{u}|^2 d x \\
   &\preceq& h \sum_{\tau \in \Omega_k^h}
   \alpha_{\tau} \sum_{x\in\partial \tau} |\hat{u}(x)|^2
 \end{eqnarray*}
where $\Omega_k^h$ is the $h$ boundary layer that is the sum of elements of $\mathcal{T}_h(\Omega_k)$ that touch (has a vertex on) the boundary $\partial\Omega_k$. We used a local inverse inequality and the discrete equivalence of the $L^2$ norm on each $\tau$. Finally, utilizing the fact that $\hat{u}$ is zero at the interior nodal points and taking maximum over $\alpha_\tau$ such that $x\in \partial \tau$ we get
$$
   a_{|\Omega_k}(u,u) \preceq h \sum_{x \in \partial\Omega_{k,h}}  \overline{\alpha}_x |u(x)|^2.
$$
The last two statements of the lemma follow directly from the first statement of the lemma plus the definitions
of the bilinear forms defined in (\ref{eq:def-b-face}) and (\ref{eq:def-b-edge}).
\end{proof}

We restate \cite[Lemma 2.2]{atle2015harmonic} which contains important estimates for the eigenfunctions found in (\ref{eq:GeF}) and (\ref{eq:GeE}).

\begin{lem}\label{lem:EigLemma}
 Let $V$ be a finite dimensional real space and consider the generalized eigenvalue problem: Find the eigenpair $(\lambda_k,\xi_k) \in \mathbb{R}\times V$ such that $b(\xi_k,\xi_k)=1$ and
 $$
   a(\xi_k,v)=\lambda_k b(\xi_k,v) \quad \forall v\in V
 $$
where the bilinear form $b(\cdot,\cdot)$ is symmetric positive definite, and the bilinear form $a(\cdot,\cdot)$ is  symmetric positive semi-definite. Then there exist $M=dim(V)$ eigenpairs with real eigenvalues ordered  as follows $0\leq \lambda_1\leq\ldots\leq \lambda_M$. 
If $\lambda_k$ is the smallest positive eigenvalue then the operator $\Pi_m:V\rightarrow V$ is defined for $k-1\leq m < M$  as
 $$
  \Pi_mu=\sum_{k=1}^m b(u,\xi_k)\xi_k
 $$ is the $b(\cdot,\cdot)$-orthogonal projection such that
\begin{equation}
 |\Pi_mv|_a\leq|v|_a\quad\text{and}\quad |v-\Pi_mv|_a \leq|v|_a, \quad \forall v\in V,
\end{equation}
and
\begin{equation}
 \|v-\Pi_mv\|_b^2\leq\frac{1}{\lambda_{m+1}}|v-\Pi_mv|_a^2, \quad \forall v\in V,
\end{equation}
where $|v|_a^2=a(v,v)$ and $\|v\|_b^2=b(v,v)$.
\end{lem}
See \cite{atle2015harmonic,Spilane:2014:ARC} for the proof.


We introduce the wire basket based coarse space interpolator or the interpolat\-ion operator $I_0^{\mathcal W}:V_h \rightarrow V_{\mathcal W} \subset V_h$ as
\begin{eqnarray}\label{eq:def-wire-coarse-int-op}
 I_0^{\mathcal W}u&:=& I_{\mathcal W} u +\sum_{\mathcal{F}_{k l}\in \mathcal{F}}
  \Pi^{\mathcal{F}_{k l}}_{m_{k l}}(u - I_{\mathcal W} u)
\end{eqnarray}
where $\Pi^{\mathcal{F}_{k l}}_{m_{k l}}:V_h \rightarrow
V_{\mathcal{F}_{k l}}^{en} \subset V_h$ is defined as follows,
\begin{eqnarray*}
 \Pi^{\mathcal{F}_{k l}}_{m_{k l}} u=
  \sum_{\mathcal{F}_{k l} \in \mathcal{F}}
  \sum_{i=1}^{m_{k l}}
  b_{\mathcal{F}_{k l}}(u|_{\mathcal{F}_{k l}},\xi^i_{\mathcal{F}_{k l}})\xi^i_{\mathcal{F}_{k l}},
\end{eqnarray*}
cf. also (\ref{eq:WC-enrich}) and (\ref{eq:def-b-face}).

We have the following lemma estimating the coarse space interpolant.
\begin{lem} \label{lem:wc-int-est}
Let  the wire basket based coarse space interpolator $I_0^{\mathcal W}$ be defined in (\ref{eq:def-wire-coarse-int-op}). Then
for $u\in V_h$
\begin{equation}
  a(u - I_0^{\mathcal W} u,u - I_0^{\mathcal W} u) \preceq \left(1+
  \max_{\mathcal{F}_{k l} \in \mathcal{F}} \frac{1}{\lambda_{\mathcal{F}_{k l}}^{m_{k l}+1}}
  \right)
  a(u,u)
\end{equation}
where $\lambda_{\mathcal{F}_{k l}}^{m_{kl}+1}$ is the $(m_{kl}+1)$th smallest eigenvalue of the generalized eigenvalue problem in Definition \ref{defGeF}. 
\end{lem}
\begin{proof}
Note that if we restrict $w=u - I_0^{\mathcal W} u$ to $\overline{\Omega}_k$ then
$\mathcal{H}_k w= \mathcal{H}_k u - I_0^{\mathcal W} u$ as $I_0^{\mathcal{W}}u$ is discrete harmonic, thus by the fact that 
$\mathcal{H}_k$ and $I-\mathcal{H}_k$ are   $a_{|\Omega_k}$  orthogonal projection in $V_h(\Omega_k)$, cf. Definition~\ref{def:DHE},
we have
\begin{eqnarray*}
 a_{|\Omega_k}(w,w)&=&a_{|\Omega_k}(w-\mathcal{H}_kw,w-\mathcal{H}_kw)+
  a_{|\Omega_k}(\mathcal{H}_kw,\mathcal{H}_kw) \\
  &\leq& a_{|\Omega_k}(u,u)+
  a_{|\Omega_k}(\mathcal{H}_kw,\mathcal{H}_kw)
\end{eqnarray*}
Thus it remains to estimate $a_{|\Omega_k}(\mathcal{H}_kw,\mathcal{H}_kw) $. By
Lemma~\ref{lem:L_Estimate} and the fact that $H_k w=w$ on $\partial\Omega_k$ we get
$$
 a_{|\Omega_k}(\mathcal{H}_kw,\mathcal{H}_kw) \preceq h\sum_{\mathcal{F}_{k l}\subset\partial\Omega_k} b_{\mathcal{F}_{k l}}(w,w).
$$

Let us now consider one such face $\mathcal{F}_{k l}\subset \partial \Omega_k$. We have $w=u-I_{\mathcal{W}}u - \Pi^{\mathcal{F}_{k l}}_{m_{k l}}(u-I_{\mathcal{W}}u)$ on the face, cf.~(\ref{eq:def-wire-coarse-int-op}).
Using Lemma~\ref{lem:EigLemma} it follows that
 \begin{eqnarray*}
  b_{\mathcal{F}_{k l}}(w,w) &=&
  \|(I-\Pi^{\mathcal{F}_{k l}}_{m_{k l}})(u-I_{\mathcal{W}}u)\|_{b_{\mathcal{F}_{k l}}}^2
 \preceq  \frac{1}{\lambda_{\mathcal{F}_{k l}}^{m_{k l}+1}}
 |(I-\Pi^{\mathcal{F}_{k l}}_{m_{k l}})(u-I_{\mathcal{W}}u)|_{a_{\mathcal{F}_{k l}}}^2
 \\
 &\preceq&  \frac{1}{\lambda_{\mathcal{F}_{k l}}^{m_{k l}+1}}
 |u-I_{\mathcal{W}}u|_{a_{\mathcal{F}_{k l}}}^2
 \end{eqnarray*}
Here $\|u\|_{b_{\mathcal{F}_{k l}}}^2:=b_{\mathcal{F}_{k l}}(u,u)$ and $|u|_{a_{\mathcal{F}_{k l}}}^2:=a_{\mathcal{F}_{k l}}(u,u)$.
Since by Definition~\ref{def:WireB-interpolant} $(I_{\mathcal{W}}u)_{|\mathcal{F}_{k l}}$ is orthogonal to $(u-I_{\mathcal{W}}u)_{|\mathcal{F}_{k l}} \in V_h^0(\mathcal{F}_{k l})$ with respect to the bilinear form $a_{\mathcal{F}_{k l}}(\cdot,\cdot)$, we have
 $$
  |u-I_{\mathcal{W}}u|_{a_{\mathcal{F}_{k l}}}^2 \leq |u|_{a_{\mathcal{F}_{k l}}}^2.
 $$
From the last two estimates, it follows that
\begin{equation} \label{eq:proof_lem_est_wire_inter:3}
h b_{\mathcal{F}_{k l}}(w,w) \preceq \frac{h}{\lambda_{\mathcal{F}_{k l}}^{m_{k l}+1}}
  |u|_{a_{\mathcal{F}_{k l}}}^2 \preceq  \frac{h}{\lambda_{\mathcal{F}_{k l}}^{m_{k l}+1}}\sum_{\tau_t \in \mathcal{T}_h(\mathcal{F}_{k l})}
 \sum_{x,y\in \partial \tau_t} \overline{\alpha}_t|u(x)-u(y)|^2.
\end{equation}
Here the last sum is over all pairs of vertices of a 2D face element $\tau_t$.
Using the definition of $\overline{\alpha}_t$ and the fact that
$|u|_{H^1(\tau)}^2\asymp \mathrm{diam}(\tau)\sum_{x,y \in \partial \tau}|u(x)-u(y)|^2$ (here $x,y$ are vertices of 3D element $\tau$)
 we get
\begin{equation} \label{eq:proof_lem_est_wire_inter}
h b_{\mathcal{F}_{k l}}(w,w) \preceq  \frac{1}{\lambda_{\mathcal{F}_{k l}}^{m_{k l}+1}}\left( a_{|\Omega_k}(u,u)+  a_{|\Omega_l}(u,u)\right).
\end{equation}
Finally, summing over the faces yields that
\begin{equation}\label{eq:proof_lem_est_wire_inter:2}
 a_{|\Omega_k}(u - I_0^{\mathcal W} u,u - I_0^{\mathcal W} u)\preceq 
 \sum_{\mathcal{F}_{k l}\subset\partial\Omega_k}  \frac{1}{\lambda_{\mathcal{F}_{k l}}^{m_{k l}+1}}
 \left( a_{|\Omega_k}(u,u)+  a_{|\Omega_l}(u,u)\right),
\end{equation}
and then summing  the subdomains ends the proof.
\end{proof}

For the next lemma we need a partition of unity. 
We need a discrete version of a partition of unity, i.e we define $\theta_i$ is a continuous function which is piecewise linear on ${\mathcal T}_h$ 
such that:
\begin{equation} \label{eq:part-unity}
\theta_i(x)=
\left\{
\begin{array}{ll}
1\quad&\forall x\in\Omega_{i,h} \\
\frac{1}{N_x}  \quad&\forall x\in\partial \Omega_{i,h}\setminus \partial\Omega_h \\
0\quad& \textrm{otherwise}\\
\end{array}.
\right.
\end{equation}
where $N_x$ is the number of suddomains $\Omega_j$ such that $x\in \overline{\Omega}_{j,h}$. As an example, 
$N_x=2$ for any nodal point $x$ on a subdomain face.

\begin{lem}\label{lem:wc-loc-est}
 Let the wire basket based coarse space interpolator $I_0^{\mathcal W}$ be defined in (\ref{eq:def-wire-coarse-int-op}) and
 let $v_k=I_h(\theta_k (u - I_0^{\mathcal W} u))$ for any $u\in V_h$ and $I_h:C(\overline{\Omega})\rightarrow V^h$ - the standard nodal 
 piecewise linear 
 interpolant. Then
\begin{equation}
  a(v_k,v_k) \preceq \left(1+
  \max_{\mathcal{F}_{k l} \subset \partial\Omega_k} \frac{1}{\lambda_{\mathcal{F}_{k l}}^{m_{k l}+1}}
    \right)
  \sum_{\partial\Omega_l\cap \partial \Omega_k=\mathcal{F}_{k l}}
  a_{|\Omega_l}(u,u).
\end{equation}
The last sum is taken over all subdomains which have a common face to $\Omega_k$.
\end{lem}
\begin{proof}
Let  $w=u - I_0^{\mathcal W} u$.
Note that
\begin{equation} \label{eq:proof_lem_loc_est}
v_k(x)=I_h \theta_k w(x)=
\left\{
\begin{array}{lr}
w(x) \quad x\in \Omega_{k,h} \\
\frac12 w(x) \quad x\in \mathcal{F}_{k l, h},  \ \mathcal{F}_{k l} \subset \partial\Omega_k\\
0 \quad x \in \mathcal{W}_h\cap\partial\Omega_k\\
0 \quad \mathrm{otherwise}
\end{array}
\right.
\end{equation}
Thus
\begin{equation} \label{eq:proof_lem_loc_est:2}
 a(v_k,v_k)=a_{|\Omega_k}(v_k,v_k)+ \sum_{\mathcal{F}_{k l}\subset \partial\Omega_k}a_{|\Omega_l}(v_k,v_k)
\end{equation}

We first estimate the second term, that is the sum of the face terms.
Note that $v_{k|\Omega_l}$ is zero at the interior nodes $\Omega_{l,h}$ and
the boundary nodes $\partial\Omega_{l,h}\setminus \mathcal{F}_{k l,h}$.
Thus by Lemma~\ref{lem:L_Estimate} and (\ref{eq:proof_lem_loc_est}) we have
$$
 \sum_{\mathcal{F}_{k l}\subset \partial\Omega_k}
   a_{|\Omega_l}(v_k,v_k) \preceq h\sum_{\mathcal{F}_{k l}\subset\partial\Omega_k} b_{\mathcal{F}_{k l}}(v_k,v_k)= \frac{h}{4}\sum_{\mathcal{F}_{k l}\subset\partial\Omega_k} b_{\mathcal{F}_{k l}}(w,w).
$$
This term has already been estimated in the proof of Lemma~\ref{lem:wc-int-est}, cf.  (\ref{eq:proof_lem_est_wire_inter}), that is
$$
h b_{\mathcal{F}_{k l}}(w,w)\preceq
 \frac{1}{\lambda_{\mathcal{F}_{k l}^{m_{k l}+1}}}\left( a_{|\Omega_k}(u,u)+  a_{|\Omega_l}(u,u)\right).
$$

We now estimate the first term in (\ref{eq:proof_lem_loc_est:2}), that is the restriction of the bilinear form to the $\Omega_k$. 
By a triangle inequality, we can write
$$
 a_{|\Omega_k}(v_k,v_k)\preceq
a_{|\Omega_k}(w,w)+ a_{|\Omega_k}(w-v_k,w-v_k)
$$
The first term has already been estimated in the proof of Lemma~\ref{lem:wc-int-est}, cf. (\ref{eq:proof_lem_est_wire_inter:2}).
Also, note that $(w-v_k)(x)$ equals either $\frac12 w(x)$ when $x$ is a face node, or
zero when $x\in\partial\Omega_k\cap \mathcal{W}_h$ and $x\in\Omega_{k,h}$.
By Lemma~\ref{lem:L_Estimate} and (\ref{eq:proof_lem_loc_est}) we thus get
$$
 a_{|\Omega_k}(w-v_k,w-v_k)\preceq h\sum_{\mathcal{F}_{k l}\subset\Omega_l} b_{\mathcal{F}_{k l}}(w-v_k,w-v_k)=
   \frac{h}{4} \sum_{\mathcal{F}_{k l}\subset\Omega_l} b_{\mathcal{F}_{k l}}(w,w).
$$
Again, this term has been estimated in the proof of Lemma~\ref{lem:wc-int-est}, cf.  (\ref{eq:proof_lem_est_wire_inter}). 
Summing all those estimates ends the proof.
\end{proof}

Analogous to the wire basket case, we now define the vertex based coarse space interpolator $I_0^{\mathcal V}:V_h \rightarrow V_{\mathcal V}\subset V_h$ as
\begin{eqnarray}\label{eq:def-vertex-coarse-int-op}
 I_0^{\mathcal V}u&:=& I_{\mathcal V} u +
 \sum_{\mathcal{F}_{k l}\in \mathcal{F}}  \Pi^{\mathcal{F}_{k l,I}}_{n_{k l}} u
  +\sum_{\mathcal{E}_k\in \mathcal{E}}   \Pi_{m_k}^{\mathcal{E}_k}(u - I_{\mathcal V} u),
\end{eqnarray}
where $\Pi^{\mathcal{F}_{k l,I}}_{n_{k l}}:V_h \rightarrow
V_{\mathcal{F}_{k l},I}^{en} \subset V_h$ and
$\Pi^{\mathcal{E}_k}_{m_k}:V_h \rightarrow
V_{\mathcal{E}_k}^{en} \subset V_h$, 
are defined as follows,
\begin{eqnarray*}
 \Pi_{n_{k l}}^{\mathcal{F}_{k l,I}}(u)&:=&
 \sum_{\mathcal{F}_{k l }\in \mathcal{F}}
 \sum_{i=1}^{n_{k l}}
 b_{\mathcal{F}_{k l}} (u|_{\mathcal{F}_{k l}},\xi^i_{\mathcal{F}_{k l,I}})\xi^i_{\mathcal{F}_{k l,I}},\\
 \Pi^{\mathcal{E}_k}_{m_k} (u)&:=&
 \sum_{\mathcal{E}_k\in \mathcal{E}}
 \sum_{i=1}^{m_k}
 b_{\mathcal{E}_k}(u|_{\mathcal{E}_k},\xi^i_{\mathcal{E}_k})\xi^i_{\mathcal{E}_k},
\end{eqnarray*}
cf. also (\ref{eq:VC-enrich}), (\ref{eq:def-b-face}) and (\ref{eq:def-b-edge}).

We have the following lemma.
\begin{lem}\label{lem:vc-int-est}
Let the vertex based coarse space interpolator $I_0^{\mathcal V}$ be defined in (\ref{eq:def-vertex-coarse-int-op}), then for any $u \in V_h$
\begin{equation}
  a(u - I_0^{\mathcal V} u,u - I_0^{\mathcal V} u) \preceq
   \left(1+
  \max\left\{
  \max_{\mathcal{F}_{k l} \in \mathcal{F}} \frac{1}{\lambda_{\mathcal{F}_{k l,I}}^{n_{k l}+1}},
  \max_{\mathcal{E}_k \in \mathcal{E}} \frac{1}{\lambda_{\mathcal{E}_k}^{m_k+1}}
  \right\}
  \right)
  a(u,u),
\end{equation}
where $\lambda_{\mathcal{F}_{k l}}^{n_{kl}+1}$ and $\lambda_{\mathcal{E}_k}^{m_k+1}$ are 
respectively the $(n_{kl}+1)$th and the $(m_k+1)$th smallest eigenvalues of the generalized eigenvalue problems in Definitions~\ref{defGeFI} and~\ref{defGeE}. 
\end{lem}
\begin{proof}
Let $w=u - I_0^{\mathcal V} u$. In the same way as in the proof of Lemma~\ref{lem:wc-int-est}, we get
\begin{eqnarray*}
 a_{|\Omega_k}(w,w)
  &\leq& a_{|\Omega_k}(u,u)+
  a_{|\Omega_k}(\mathcal{H}_kw,\mathcal{H}_kw)
\end{eqnarray*}
Next we bound $a_{|\Omega_k}(\mathcal{H}_kw,\mathcal{H}_kw) $.
By Lemma~\ref{lem:L_Estimate}, cf. (\ref{eq:lem:L_Estimate}), and  (\ref{eq:def-b-face}) and (\ref{eq:def-b-edge}). we get
\begin{eqnarray*}
 a_{|\Omega_k}(\mathcal{H}_kw,\mathcal{H}_kw) & \preceq&
  h^2\sum_{\mathcal{E}_j\subset\Omega_k} b_{\mathcal{E}_j}(w,w)+
  h\sum_{\mathcal{F}_{k l}\subset\Omega_k} b_{\mathcal{F}_{k l}}(w,w)
\end{eqnarray*}
Note that by (\ref{eq:vertex-coarse-sp}) and  Definition~\ref{def:vertex-interpolant} we have $w_{|\mathcal{E}_j}=u-  I_{\mathcal V} u - \Pi_{m_j}^{\mathcal{E}_j}(u - I_{\mathcal V} u)$ for any edge $\mathcal{E}_j \subset \mathcal{W}\cap\partial\Omega_k$, and $w_{|\mathcal{F}_{k l}}=u- \Pi_{m_{k l}}^{\mathcal{F}_{k l,I}}(u)$ for any face $\mathcal{F}_{k l} \subset \mathcal{W}\cap\partial\Omega_k$.

Now, consider the term $b_{\mathcal{E}_{j}}(w,w)$ related to an edge $\mathcal{E}_j$.
By Lemma~\ref{lem:EigLemma} we get
 \begin{eqnarray*}
  b_{\mathcal{E}_j}(w,w)
 &=&
  \|(I-\Pi^{\mathcal{E}_j}_{m_j})(u-I_{\mathcal{V}}u)\|_{b_{\mathcal{E}_j}}^2
 \preceq  \frac{1}{\lambda_{\mathcal{E}_j}^{m_j+1}}
 |(I-\Pi^{\mathcal{E}_j}_{m_j})(u-I_{\mathcal{V}}u)|_{a_{\mathcal{E}_j}}^2
 \\
 &\preceq&  \frac{1}{\lambda_{\mathcal{E}_j}^{m_j+1}}
 |u-I_{\mathcal{V}}u|_{a_{\mathcal{E}_j}}^2
 \preceq \frac{1}{\lambda_{\mathcal{E}_j}^{m_j+1}}
 |u|_{a_{\mathcal{E}_j}}^2.
 \end{eqnarray*}
  Here $\|u\|_{b_{\mathcal{E}_j}}^2:=b_{\mathcal{E}_j}(u,u)$ and
  $|u|_{a_{\mathcal{E}_j}}^2:=a_{\mathcal{E}_j}(u,u)$.
We also used the fact that $(I_{\mathcal{V}}u)_{\mathcal{E}_j}$ is orthogonal to $(u-I_{\mathcal{V}}u)_{|\mathcal{E}_j} \in V_h^0(\mathcal{E}_{j}^0)$ with respect to the bilinear form $a_{\mathcal{E}_j}(\cdot,\cdot)$, cf. Definition~\ref{def:vertex-interpolant}.

Further, using the fact that, for $u$ linear, $|u|_{H^1(\tau_e)}^2$ is equivalent to $h^{-1}|u(x)-u(y)|^2$ ($x,y$ are the ends of the 1D element $\tau_e$), we get
$$
 h^2b_{\mathcal{E}_j}(w,w)\preceq \frac{h^2}{\lambda_{\mathcal{E}_j}^{m_j+1}}
 |u|_{a_{\mathcal{E}_j}}^2 \preceq \frac{h}{\lambda_{\mathcal{E}_j}^{m_j+1}}\sum_{\tau_e \in \mathcal{T}_h(\mathcal{E}_j)}
 \sum_{x,y\in \partial \tau_e} \overline{\alpha}_e|u(x)-u(y)|^2.
$$
Here the last sum is over the ends $x,y$ of an 1D edge element $\tau_e$.
Note that $h|u(x)-u(y)|^2 \preceq \int_\tau |\nabla u|^2\: dx $ if $x,y$ are vertices of $\tau\in \mathcal{T}_h$.
Thus  we get
$$
 h^2
 b_{\mathcal{E}_j}(w,w)\preceq 
\sum_{\partial\tau \cap \mathcal{E}_j=\overline{\tau}_e \subset\mathcal{E}_j}
\int_{\tau}\alpha|\nabla u|^2\;d x.
$$
where the sum is over all 3D fine elements such that one of its edges is contained in $\mathcal{E}_j$.

Now, consider the term $b_{\mathcal{F}_{k l}}(w,w)$ related to a face $\mathcal{F}_{k l}$.
Note that $u_{|\partial\mathcal{F}_{k l}}$ does not have to be equal to zero in general but if we define a function
$\hat{u}$ such that $\hat{u}(x)=u(x)$ $x\in\mathcal{F}_{kl,h}$ and $\hat{u}(x)=0$
for $x\in \partial\mathcal{F}_{k l,h}$, then we have
\begin{equation}
 b_{\mathcal{F}_{k l}}(u,u)=b_{\mathcal{F}_{k l}}(\hat{u},\hat{u}), \qquad
 a_{\mathcal{F}_{k l,I}}(u,u)=a_{\mathcal{F}_{k l,I}}(\hat{u},\hat{u}),
\end{equation}
cf. (\ref{eq:def-b-face}) and (\ref{eq:def-a-f_kl-I}).
Thus we see that $\Pi^{\mathcal{F}_{k l,I}}_{n_{k l}}u=\Pi^{\mathcal{F}_{k l,I}}_{n_{k l}}\hat{u}$ and
we can apply Lemma~\ref{lem:EigLemma} replacing $u$ with $\hat{u}$ and get
 \begin{eqnarray*}
  b_{\mathcal{F}_{k l}}(w,w) &=&
  \|(I-\Pi^{\mathcal{F}_{k l,I}}_{n_{k l}})u\|_{b_{\mathcal{F}_{k l}}}^2=  \|(I-\Pi^{\mathcal{F}_{k l,I}}_{n_{k l}})\hat{u}\|_{b_{\mathcal{F}_{k l}}}^2
 \preceq  \frac{1}{\lambda_{\mathcal{F}_{k l}}^{n_{k l}+1}}
 |(I-\Pi^{\mathcal{F}_{k l}}_{n_{k l}})\hat{u}|_{a_{\mathcal{F}_{k l,I}}}^2
 \\
 &\preceq&  \frac{1}{\lambda_{\mathcal{F}_{k l,I}}^{n_{k l}+1}}
 |\hat{u}|_{a_{\mathcal{F}_{k l,I}}}^2= \frac{1}{\lambda_{\mathcal{F}_{k l,I}}^{n_{k l}+1}}
 |u|_{a_{\mathcal{F}_{k l,I}}}^2 \leq \frac{1}{\lambda_{\mathcal{F}_{k l,I}}^{n_{k l}+1}}
 |u|_{a_{\mathcal{F}_{k l}}}^2. 
 \end{eqnarray*}
  Here
  $|u|_{a_{\mathcal{F}_{k l,I}}}^2:=a_{\mathcal{F}_{k l,I}}(u,u)$.
  The last inequality follows from the fact that 
$$
  a_{\mathcal{F}_{k l,I}}(v,v)\leq a_{\mathcal{F}_{k l,}}(v,v) \qquad \forall v\in V_h^0(\mathcal{F}_{k l})
$$
cf.  (\ref{eq:aFkl}) and (\ref{eq:def-a-f_kl-I}).
Further, analogously as in the proof of Lemma~\ref{lem:wc-int-est}, cf.
(\ref{eq:proof_lem_est_wire_inter:3}) and (\ref{eq:proof_lem_est_wire_inter}), we get
$$
  hb_{\mathcal{F}_{k l}}(w,w)\preceq
 \frac{1}{\lambda_{\mathcal{F}_{k l,I}^{n_{k l}+1}}}\left( a_{|\Omega_k}(u,u)+  a_{|\Omega_l}(u,u)\right).
$$
Finally, by summing over the edges and faces, and then over the subdomains we end the proof.
\end{proof}
\begin{lem}\label{lem:vc-loc-est}
Let the vertex based coarse space interpolator $I_0^{\mathcal V}$ be defined in (\ref{eq:def-vertex-coarse-int-op}), and $v_k=I_h(\theta_k (u - I_0^{\mathcal V}u))$ for any $u\in V_h$. Then
\begin{equation}
  a(v_k,v_k) \preceq  \left(1+
  \max\left\{
  \max_{\mathcal{F}_{k l} \subset \partial\Omega_k} \frac{1}{\lambda_{\mathcal{F}_{k l,I}}^{n_{k l}+1}},
  \max_{\mathcal{E}_k \subset \partial\Omega_k} \frac{1}{\lambda_{\mathcal{E}_k}^{m_k+1}}
  \right\}
   \right)
    \sum_{\mathcal{E}_s\subset \partial\Omega_l\cap \partial \Omega_k}
   a_{|\Omega_l}(u,u).
\end{equation}
The last sum is taken over all subdomains which share an edge with $\Omega_k$.
\end{lem}
\begin{proof}
Let $w=u-I_0^{\mathcal V}u$, then we have $I_h\theta_k w$ equal to $w$ at interior nodes $\Omega_{k,h}$, $\frac12 w$ at the nodes $\mathcal{F}_{k l,h}$ on each face $\mathcal{F}_{k l}$ of $\Omega_k$, $\frac{1}{n(\mathcal{E}_i)}w$ at the nodes $\mathcal{E}_{i,h}$ on each edge $\mathcal{E}_i$  of $\Omega_k$, and zero at all remaining nodal points of $\Omega_h$. Here $n(\mathcal{E}_i)$ is the number of subdomains that share the edge $\mathcal{E}_i$.
As in the proof of Lemma~\ref{lem:wc-loc-est}, we can write
\begin{eqnarray*}
 a(v_k,v_k)&=&a_{|\Omega_k}(v_k,v_k)+
 \sum_{\mathcal{E}_i\subset \partial\Omega_k}a_{|\Omega_i}(v_k,v_k) +
 \sum_{\mathcal{F}_{k l}\subset \partial\Omega_k} a_{|\Omega_l}(v_k,v_k)
\\
&\preceq&a_{|\Omega_k}(v_k,v_k)+
h^2\sum_{\mathcal{E}_i\subset\partial\Omega_k} b_{\mathcal{E}_i}(w,w)+
h\sum_{\mathcal{F}_{k l}\subset\partial\Omega_k} b_{\mathcal{F}_{k l}}(w,w)).
\end{eqnarray*}
The face and the edge terms can be estimated following the lines in the proof of Lemma~\ref{lem:vc-int-est}.

The first term, on the other hand, can be estimated following the lines in the proof of Lemma~\ref{lem:wc-loc-est}, that is using Lemma~\ref{lem:L_Estimate}, as follows:
\begin{eqnarray*}
a_{|\Omega_k}(v_k,v_k)&\preceq &
a_{|\Omega_k}(w,w)+ a_{|\Omega_k}(w-v_k,w-v_k)
\\
&\preceq &
a_{|\Omega_k}(w,w) \\
&& + h^2\sum_{\mathcal{E}_i\subset\partial\Omega_k} b_{\mathcal{E}_i}(v_k-w,v_k-w)+h\sum_{\mathcal{F}_{k l}\subset\partial\Omega_k} b_{\mathcal{F}_{k l}}(v_k-w,v_k-w)
\\
&\preceq &
a_{|\Omega_k}(w,w)
+ h^2\sum_{\mathcal{E}_i\subset\partial\Omega_k} b_{\mathcal{E}_i}(w,w)+h\sum_{\mathcal{F}_{k l}\subset\partial\Omega_k} b_{\mathcal{F}_{k l}}(w,w).
\end{eqnarray*}
Again, these terms can be estimated in the same way as in the proof of Lemma~\ref{lem:vc-int-est}. 

Combining those estimates, we get the proof.
\end{proof}

The next and final lemma, provides estimates for the stability of decomposition for the two preconditioners presented in this paper, which are required in the proof of Theorem \ref{thm:ASM-cond-est}. 

\begin{lem}[Stable Decomposition]\label{lem:StableDecomp}
Let $k=1$ and $k=2$ in the superscript refer to the two coarse spaces: the wire basket based coarse space and the vertex based coarse space, respectively. Then, for all $u\in V_h$ there exists a representation $u=u_0^{(k)}+\sum^N_{i=1} u_i^{(k)}$
such that $u_0^{(k)}\in V_0^{(k)}$, and $u_i^{(k)}\in V_i$, $i=1,\ldots,N$, and
\begin{equation}
 a(u_0^{(k)},u_0^{(k)})+\sum_{i=1}^N a(u_i^{(k)},u_i^{(k)}) \preceq
\left(C_0^{(k)}\right)^2 a(u,u),
\end{equation}
for $k=1, 2$, with
\begin{eqnarray*}
  \left(C_0^{(1)}\right)^2 &=& 1+\max_{\mathcal{F}_{k l} \in \mathcal{F}} \left(\lambda_{\mathcal{F}_{k l}}^{m_{k l}+1}\right)^{-1}
  \\
  \left(C_0^{(2)}\right)^2 &=& 1+
  \max\left\{\max_{\mathcal{F}_{k l} \in \mathcal{F}} (\lambda_{\mathcal{F}_{k l,I}}^{n_{k l}})^{-1},
        \min_{\mathcal{E}_k \in \mathcal{E}} (\lambda_{\mathcal{E}_k}^{m_k})^{-1}\right\}
\end{eqnarray*}
\end{lem}
\begin{proof}
For the wire basket based coarse space, let $u_0^{(1)}=I_0^{\mathcal W}u$ and $u_i^{(1)}=I_h(\theta_i(u-I_0^{\mathcal W}u))$.
The corresponding statement of the lemma then follows immediately from the Lemmas~\ref{lem:wc-int-est} and~\ref{lem:wc-loc-est}.

For the vertex based coarse space, let $u_0^{(2)}=I_0^{\mathcal V}u$ and $u_i^{(2)}=I_h(\theta_i(u-I_0^{\mathcal V}u))$.
It's statement of the lemma then follows immediately from the Lemmas~\ref{lem:vc-int-est} and~\ref{lem:vc-loc-est}.
\end{proof}

\noindent
{\bf Proof of Theorem~\ref{thm:ASM-cond-est}}
\begin{proof} The convergence theory of the abstract Schwarz framework indicates that, under three assumptions, the condition number of our method can be bounded as the following.
\begin{equation}
 \kappa(P^{(k)})\leq (C_0^{(k)})^2\omega(\rho+1).
\end{equation}
The three assumptions are concerned with estimating the three constants $\omega$, $\rho$ and $C^2_0$. It is easy to that $\omega=1$ here, as we have nested subspaces and are using exact solvers for the subproblems, and $\rho\leq N_c$, where $N_c$ is the maximum number of subspaces that cover any $x\in\Omega$. We refer to \cite[section 5.2]{smith2004domain} or \cite[section 2.3]{toselli2005domain} for details. An estimate of the last parameter $(C_0^{(k)})^2$, for $k=1$ or $k=2$, is given in Lemma~\ref{lem:StableDecomp}. 
The proof of the theorem now follows.
\end{proof} 

\section{Numerical results}\label{sec:numerical_results}
For the numerical experiment we choose our model elliptic problem to be defined on a unit cube, with homogeneous boundary condition and $f(x) = 100$. For the coefficient $\alpha$, we chose distributions with channels and inclusions across subdomain boundaries (faces and edges), cf. Fig.~\ref{fig:fourfaces} and Fig.~\ref{fig:twoedges}.   
The preconditioned conjugate gradient method is used to solve the system with our proposed additive Schwarz methods as preconditioners based on the vertex based coarse space and the wirebasket based coarse space, requiring the iteration to stop when the residual norm is reduced by the factor $10^{-6}$. The domain (unit cube) is partitioned into $2\times 2\times 2$ subdomains. Further discretization is done by dividing each subdomain into small cubes of size $h\times h\times h$, and then dividing the small cubes into six tetrahedra each.

The algorithms have been implemented in {\tt MATLAB}, using {\tt meshgrid} to mesh the cube and {\tt delaunayTriangulation} to triangulate the mesh, creating equal number of tetrahedrons in each direction, and routines from {\tt PDEToolbox} for assembling the stiffness and mass matrices. For the iterative solver, {\tt pcgeig}, an extension of {\tt MATLAB}'s preconditioned conjugate gradient routine {\tt pcg}, available from {\it m2matlabdb.ma.tum.de}, has been used. This routine, in addition to finding the solution, returns an estimate of the condition number for the preconditioned system, which are reported here. The two main ingredients in constructing the spectral components of our coarse spaces, are the solution of their corresponding eigenvalue problems and the discrete harmonic extensions inside subdomains. Built-in function {\tt eigs} has been used for the generalized eigenvalue problems. 

Eigenvalues below a given threshold $\lambda^*$ (cf. Theorem \ref{thm:ASM-cond-est}) are chosen for the construction of the spectral component of our coarse spaces with their corresponding eigenvectors, consequently, enrichments are only applied to those faces and edges that are affected by the inclusions. The larger the threshold $\lambda^*$ is the higher the number of eigenfunctions used. There are no particular guidelines in the literature for the optimal choice, we refer to for instance the work of \cite{klawonn:2016:adaptive} for other choices. 
For our experiment we choose the threshold proportional to $\frac Hh$. One can also choose the smallest eigenvalue of the corresponding problem with constant coefficients as the threshold, which will then ensure that no eigenfunctions be used when $\alpha$ is constant.  

\begin{figure}[htp]
\centering
 \includegraphics[width=0.24\linewidth]{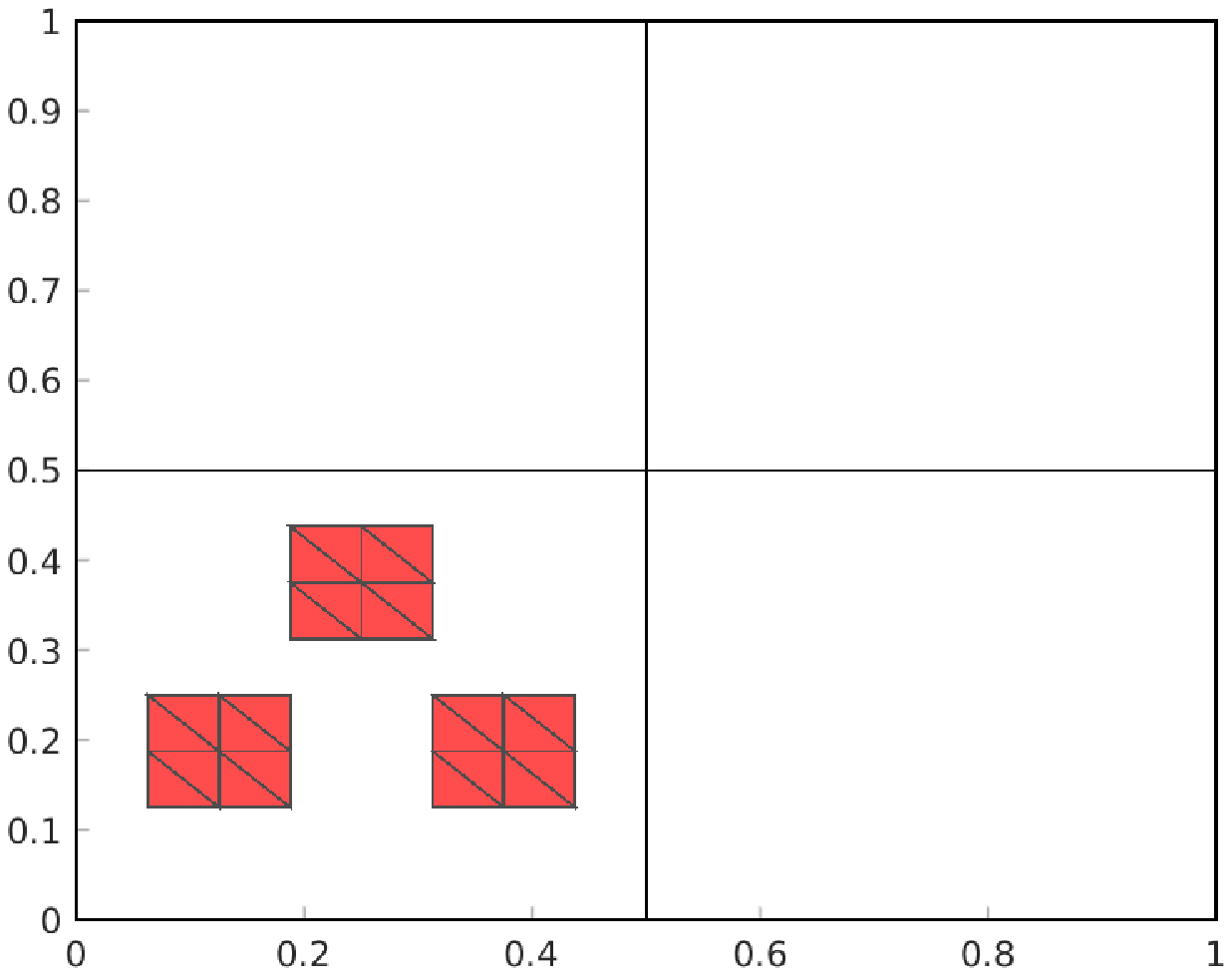}
 \includegraphics[width=0.24\linewidth]{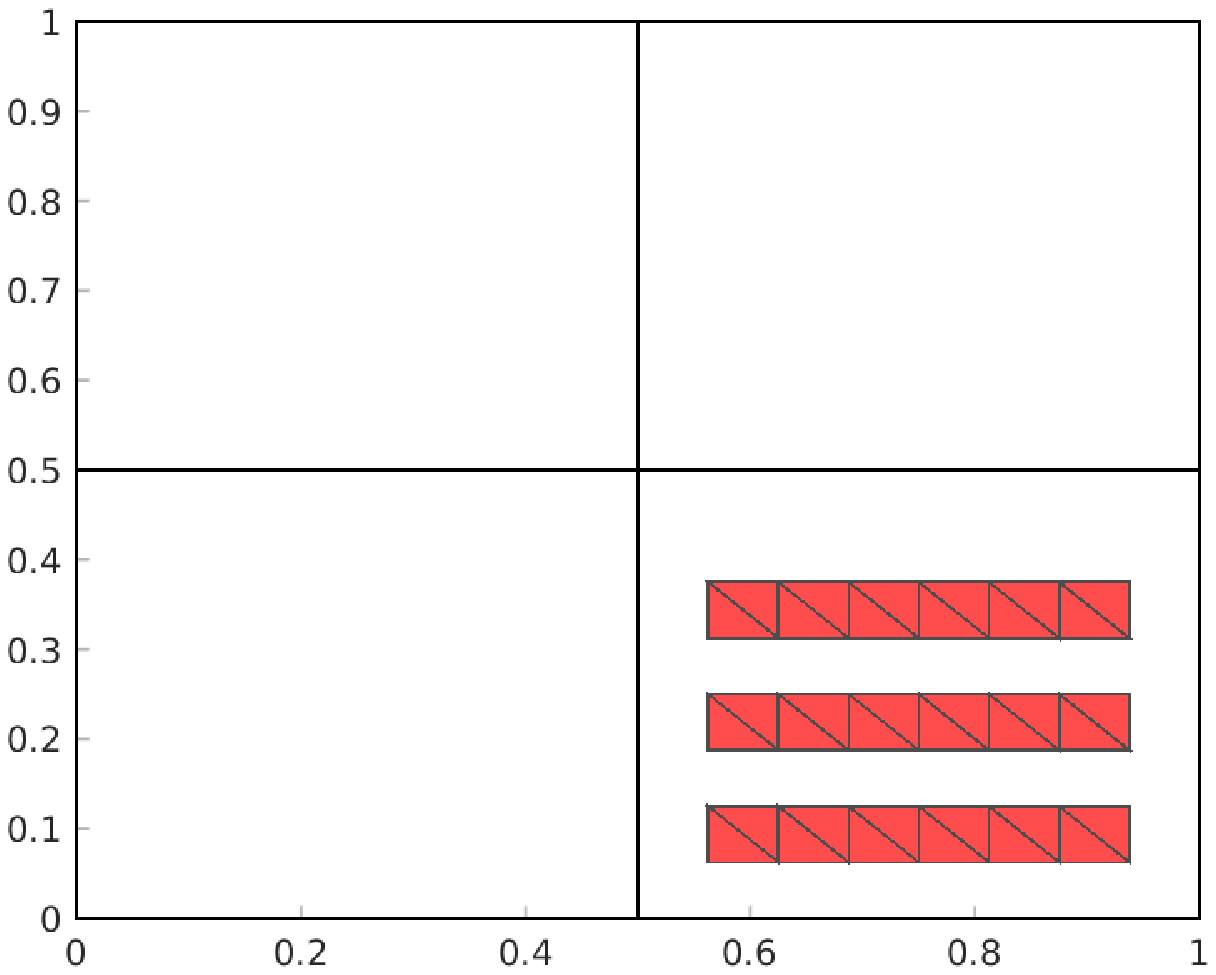}
 \includegraphics[width=0.24\linewidth]{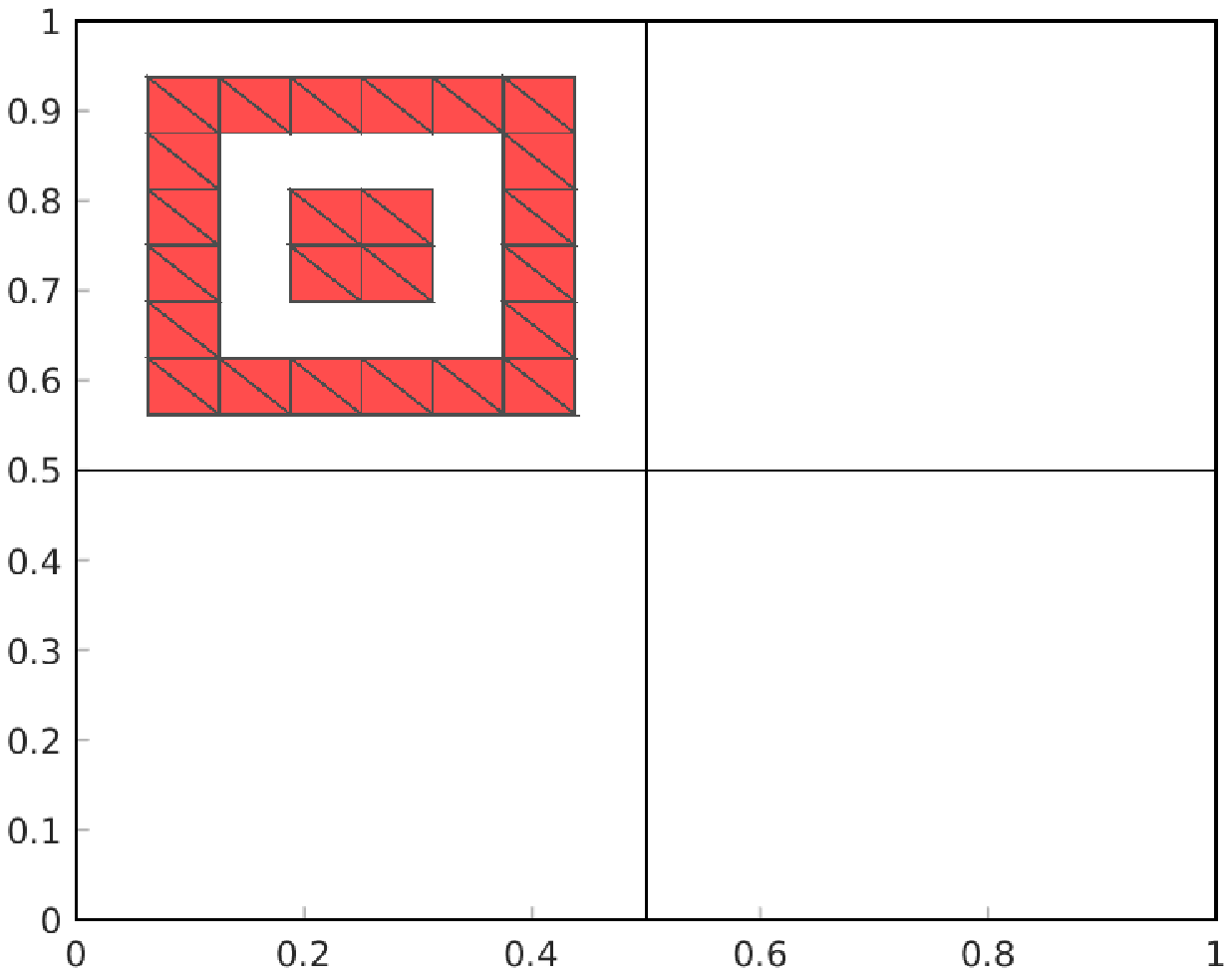}
 \includegraphics[width=0.24\linewidth]{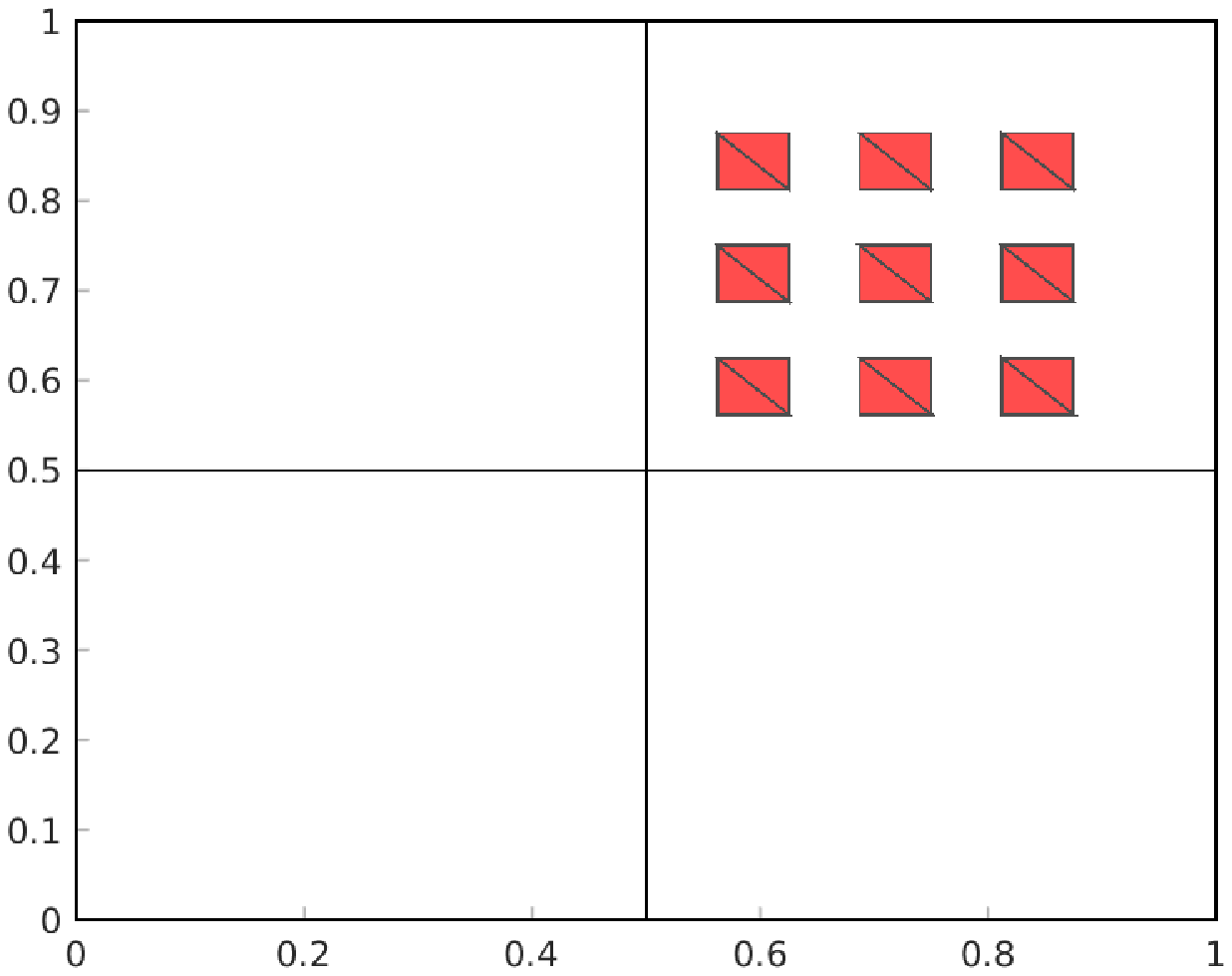}
 \includegraphics[width=0.49\linewidth]{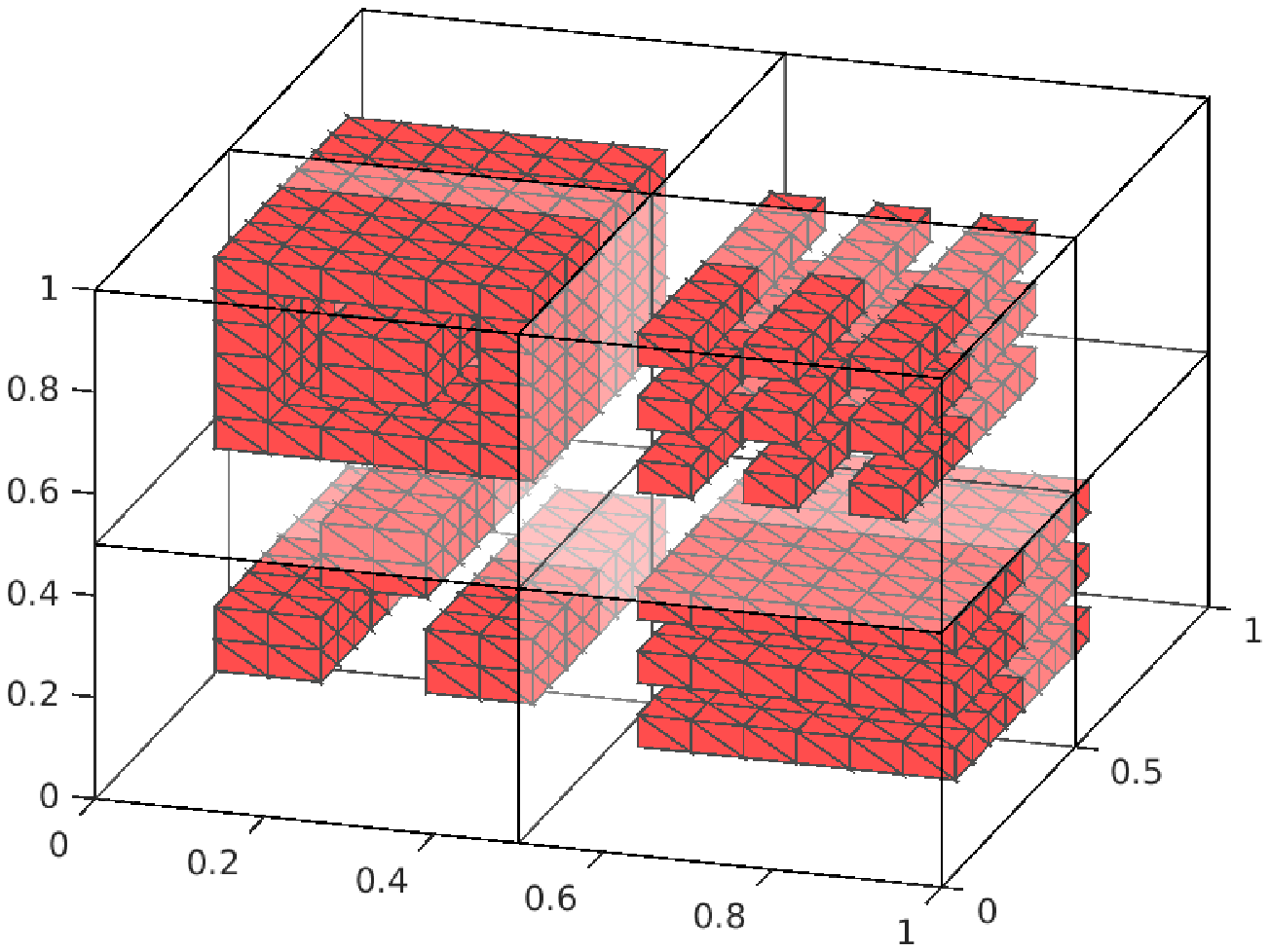}
 \includegraphics[width=0.49\linewidth]{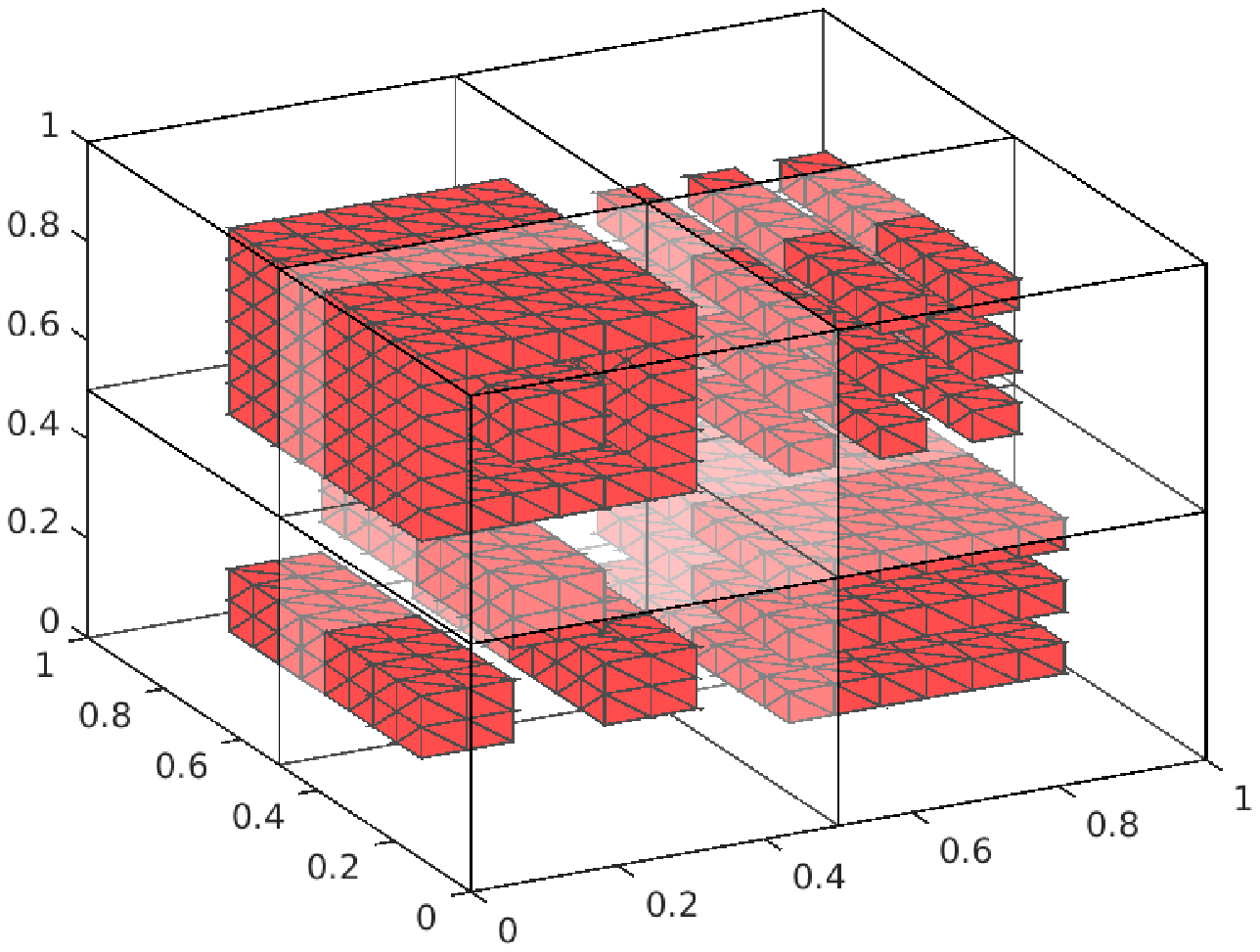}
 \caption{Showing eight ($2\times 2 \times 2$) subdomains, with inclusions (regions with large coefficients, $\alpha=1\mathrm{e}{6}$, shown in red) through faces in the $xz-$plane. The faces, as shown from left to right in the first row, are called Face 1, Face 2, Face 3, and Face 4. For the ease of illustrating, the distribution is shown from two different angles in the second row.}
\label{fig:fourfaces}
\end{figure}

\begin{table}[htp]
\begin{center}
\begin{tabular}{c|c|ccc}
\multicolumn{5}{c}{Vertex based coarse space} \\ 
 &No enrichment &\multicolumn{3}{c}{Face enrichment with $\lambda_{\mathcal{F}_I} < \lambda^*$} \\
 && $\lambda^*=0.0375$	&$\lambda^*=0.0187$			&$\lambda^*=0.0094$ \\
 $\alpha$ distribution &$H/h=8$		&$H/h=8$		&$H/h=16$			&$H/h=32$\\
 \hline 
 \\
Face 1		&$3.53\mathrm{e}{5} ~ (41)$		&$20.87 ~ (23)$			&$39.27 ~ (29)$			&$75.41 ~ (41)$	\\
Face 2		&$1.90\mathrm{e}{5} ~ (40)$		&$19.72 ~ (22)$			&$37.08 ~ (29)$			&$70.75 ~ (42)$	\\
Face 3		&$2.22\mathrm{e}{5} ~ (42)$		&$18.48 ~ (28)$			&$32.60 ~ (30)$			&$62.00 ~ (41)$	\\
Face 4		&$2.13\mathrm{e}{5} ~ (114)$		&$21.46 ~ (23)$			&$38.62 ~ (30)$			&$72.00 ~ (39)$	\\
 \hline
Face 1-4	&$3.42\mathrm{e}{5} ~ (170)$		&$18.49 ~ (29)$			&$24.52 ~ (30)$			&$45.29 ~ (39)$	\\
 \multicolumn{5}{c}{}\\
 \multicolumn{5}{c}{Wirebasket based coarse space}\\
 &No enrichment &\multicolumn{3}{c}{Face enrichment with $\lambda_{\mathcal{F}} < \lambda^*$} \\
 && $\lambda^*=0.075$	&$\lambda^*=0.035$			&$\lambda^*=0.0187$ \\
 $\alpha$ distribution &$H/h=8$		&$H/h=8$			&$H/h=16$			&$H/h=32$\\
 \hline 
 \\		
Face 1		&$1.14\mathrm{e}{4} ~ (34)$		&$12.56 ~ (19)$			&$19.55 ~ (24)$			&$34.92 ~ (32)$	\\
Face 2		&$9.23\mathrm{e}{1} ~ (32)$		&$12.84 ~ (21)$			&$20.36 ~ (30)$			&$35.18 ~ (33)$	\\
Face 3		&$2.50\mathrm{e}{5} ~ (42)$		&$12.93 ~ (22)$			&$20.37 ~ (27)$			&$35.55 ~ (33)$	\\
Face 4		&$1.21\mathrm{e}{5} ~ (79)$		&$12.90 ~ (20)$			&$19.54 ~ (25)$			&$34.75 ~ (33)$	\\
 \hline
Face 1-4	&$2.50\mathrm{e}{5} ~ (152)$		&$12.41 ~ (24)$			&$19.52 ~ (32)$			&$30.65 ~ (34)$	\\ 
 \\
\end{tabular}
\end{center}
\caption{Condition number estimates and iteration counts (inside brackets) for the two preconditioners, the vertex based and the wirebasket based preconditioner. The rows correspond to the different $\alpha$ distribution as shown in Fig.~\ref{fig:fourfaces}, and the columns correspond to no enrichment and with enrichment.}\label{tab:results_on_faces}
\end{table}

In our first experiment, we apply our methods to the model problems with high conductivity inclusions cutting through faces, as shown in Fig.~\ref{fig:fourfaces}, where the high conductivity regions (cells) are shown in red, with $\alpha=10^6$ in the inclusions and $\alpha=1$ otherwise. Condition number estimate and the number of iteration required to converge in each test case, are given in Table~\ref{tab:results_on_faces}. Each row of the table corresponds to a distribution, which includes inclusion through only one of the four faces, Face 1 to 4, cf. Fig.~\ref{fig:fourfaces}, or through all four faces. of the columns, the first one correspond to no enrichment, while the other three correspond to enrichment with different mesh sizes (keeping $H$ fixed, while varying $h$). The size of each inclusion, that is its width and length, has been kept the same throughout. As seen from the table, both adaptive enrichment of the coarse space does improve the performance of the iterative method, the condition number varies like $\frac Hh$ as $\lambda^*$ varies like $\frac hH$, plus the different distributions with face inclusions somehow do not affect the performance. 

\begin{table}[htp]
\begin{center}
\begin{tabular}{cccccccc}
\multicolumn{2}{c}{Face 1} 					& \multicolumn{2}{c}{Face 2} 					& \multicolumn{2}{c}{Face 3} 					& \multicolumn{2}{c}{Face 4} \\
\hline
$\lambda_{\mathcal{F}_I}$ 	& $\lambda_{\mathcal{F}}$ 	& $\lambda_{\mathcal{F}_I}$ 	& $\lambda_{\mathcal{F}}$ 	& $\lambda_{\mathcal{F}_I}$ 	& $\lambda_{\mathcal{F}}$ 	& $\lambda_{\mathcal{F}_I}$ 	& $\lambda_{\mathcal{F}}$ 	\\
  \hline	
	$0.0$			&$6.1\mathrm{e}{-8}$ 		& $0.0$				&$5.3\mathrm{e}{-8}$		& $0.0$				&$5.0\mathrm{e}{-8}$		&$0.0$				& $6.5\mathrm{e}{-8}$ 		\\
	$6.1\mathrm{e}{-8}$	&$1.1\mathrm{e}{-7}$		& $5.6\mathrm{e}{-8}$		&$1.2\mathrm{e}{-7}$ 		& $1.6\mathrm{e}{-7}$		&$1.7\mathrm{e}{-7}$ 		&$7.9\mathrm{e}{-8}$		& $1.6\mathrm{e}{-7}$ 		\\
	$7.5\mathrm{e}{-8}$	&$1.2\mathrm{e}{-7}$		& $1.7\mathrm{e}{-7}$		&$2.0\mathrm{e}{-7}$		& $5.4\mathrm{e}{-3}$		&$5.4\mathrm{e}{-3}$		&$7.9\mathrm{e}{-8}$		& $1.6\mathrm{e}{-7}$		\\
$\mathbf{2.5\mathrm{e}{-2}}$	&$\mathbf{1.1\mathrm{e}{-1}}$ 	& $\mathbf{1.3\mathrm{e}{-2}}$	&$1.3\mathrm{e}{-2}$		& $5.4\mathrm{e}{-3}$		&$5.4\mathrm{e}{-3}$		&$1.6\mathrm{e}{-7}$		& $2.4\mathrm{e}{-7}$		\\
				&				& 				&$1.3\mathrm{e}{-2}$		& $\mathbf{1.9\mathrm{e}{-2}}$	&$\mathbf{1.9\mathrm{e}{-2}}$	&$2.3\mathrm{e}{-7}$		& $2.6\mathrm{e}{-7}$		\\
				&				&				&$1.3\mathrm{e}{-2}$		&				&				&$2.3\mathrm{e}{-7}$		& $2.6\mathrm{e}{-7}$		\\
				&				&				&$\mathbf{5.0\mathrm{e}{-2}}$	&				&				&$2.8\mathrm{e}{-7}$		& $3.2\mathrm{e}{-7}$ 		\\
				&				&				&				&				&				&$2.9\mathrm{e}{-7}$		& $3.2\mathrm{e}{-7}$		\\
				&				&				&				&				&				&$3.7\mathrm{e}{-7}$		& $3.8\mathrm{e}{-7}$		\\
				&				&				&				&				&				&$\mathbf{2.8\mathrm{e}{-2}}$	& $\mathbf{1.4\mathrm{e}{-1}}$	\\		
\end{tabular}
\end{center}
\caption{Showing the eigenvalues for $H/h=32$ and the different faces, which (except for the one in boldface) are below the threshold and have been selected for the enrichment, and the one (in boldface) which is just above the threshold, cf. Table \ref{tab:face eigenvalues}. $\lambda_{\mathcal{F}_I}$ and $\lambda_{\mathcal{F}}$ represent the eigenvalues of the face eigenvalue problems of the vertex based and the wirebasket based coarse space, respectively.} 
\label{tab:face eigenvalues}
\end{table}

To have an idea of the size of the enrichment, that is the number of eigenfunctions used to enrich the coarse space, we have listed their eigenvalues for the largest problem $\frac{H}{h} = 32$, in Table \ref{tab:face eigenvalues} where each pair of columns represents a face. The eigenvalues selected are listed above the eigenvalue printed in boldface which is the first eigenvalue not used in the enrichment. As we can see from the table the number of eigenfunctions needed in each case are small. Normally, there is a clear separation between the eigenvalues selected for the enrichment and the one not selected, which is visible through the jump in their values which is proportional to the contrast in $\alpha$, although in some cases (e.g. Face 2 and -3 here) a few eigenvalues from the other side of the jump may get selected. Note that the first eigenvalue of the face eigenvalue problem for the vertex based coarse space is always zero, and is always selected for the enrichment. Since there are no inclusions cutting through edges here, no edge enrichment is needed.   

In our next experiment, we apply our methods to the model problems with high conductivity inclusions cutting through edges, cf. Fig.~\ref{fig:twoedges}, for which we have chosen two rectangular slabs with $\alpha=10^6$ as inclusions placed horizontally, which is illustrated by their projections on to the $xz-$ and $yz$ plane. As we can see, they not only cut the two vertical edges, but also the eight vertical faces, requiring to solve eigenvalue problems both on edges and faces, which due to symmetry reduces to solving an identical eigenvalue problem on the edges, a second identical eigenvalue problem on the $xz$ faces, and a third identical eigenvalue problem on the $yz$ faces. The corresponding eigenvalues for the problem $\frac Hh = 32$ are given in Table \ref{tab:edge eigenvalues}. The condition number estimates and the iterations counts are given in Table~\ref{tab:results_on_edges}. we have provided results for two test cases, one with the edge inclusions, and one with both edge and face inclusions together. As seen from the table, enrichments are needed to improve the convergence of the iteration in all cases, except for the wirebasket based coarse space with edge inclusion alone, in which case, there is no need for enrichment. This is obvious since the wirebasket based coarse space restricted to an edge is a full space.    

\begin{figure}[htp]
\centering
 \includegraphics[width=0.3\linewidth]{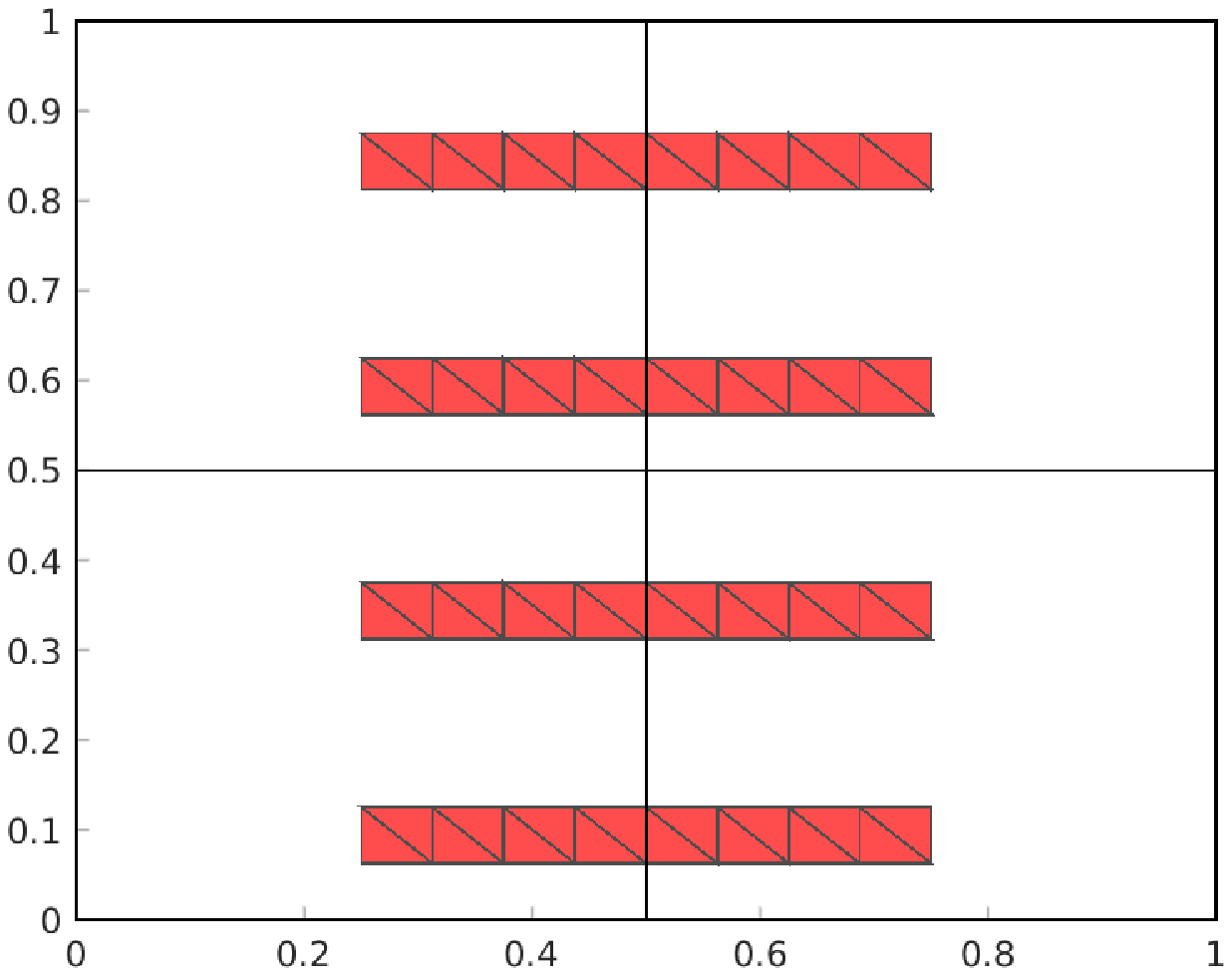}
 \includegraphics[width=0.3\linewidth]{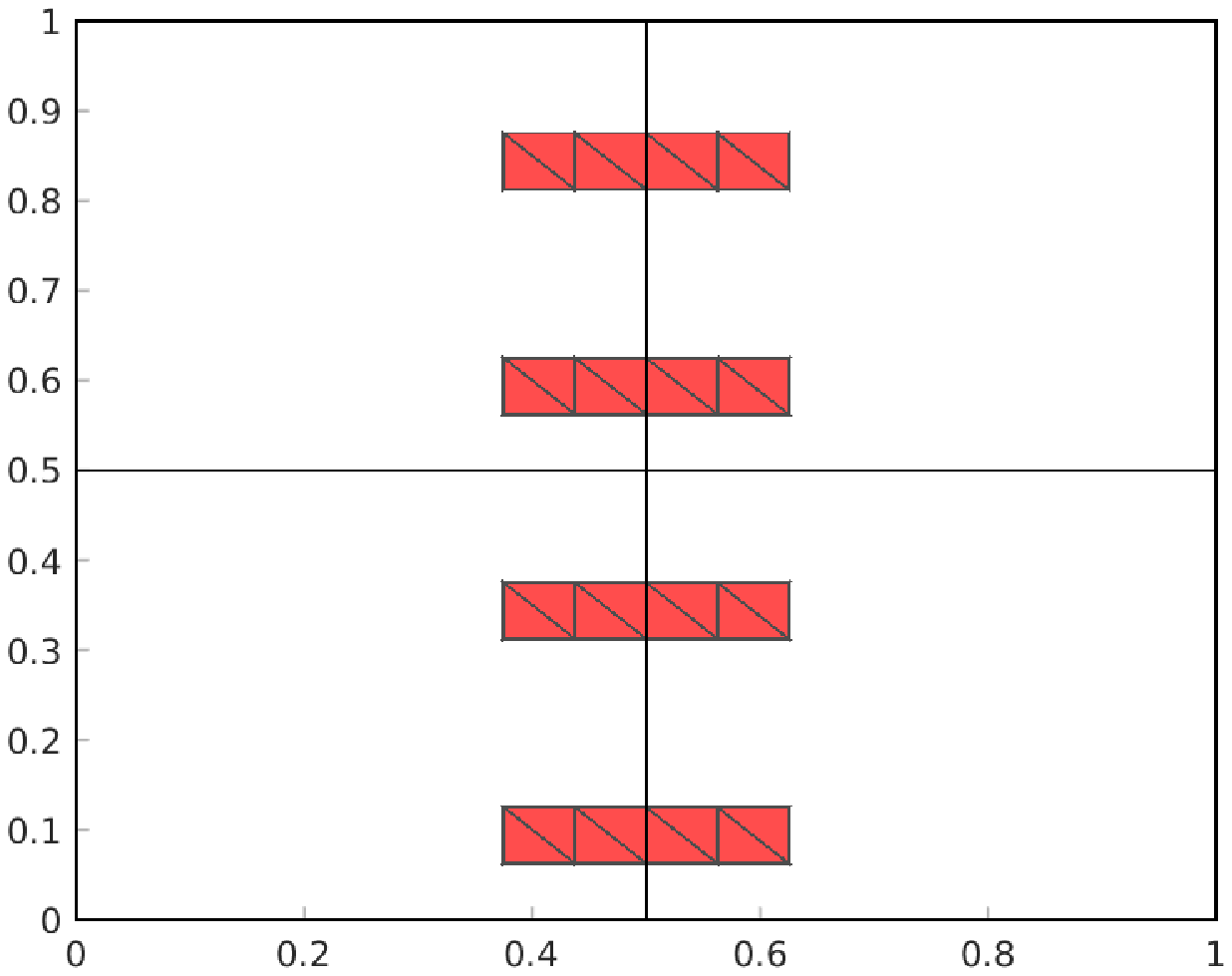}
 \caption{Inclusions with large coefficient cutting through the vertical edges (regions with large coefficients, $\alpha=10^6$, shown in red) are illustrated by their projections on to the $xz$ plane (left figure) and the $yz$ plane (right figure)}
 \label{fig:twoedges}
\end{figure}

\begin{table}[htp]
\begin{center}
\begin{tabular}{c|c|ccc}
\multicolumn{5}{c}{Vertex based coarse space} \\ 
 &No enrichment &\multicolumn{3}{c}{Edge-, Face enrichment: $\lambda_{\cal E} < \lambda^*$, $\lambda_{\mathcal{F}_I} < \lambda^*$} \\
 && $\lambda^*=0.1512$ & $\lambda^*=0.0756$ & $\lambda^*=0.0378$ \\
 $\alpha$ distribution &$H/h=8$		&$H/h=8$		&$H/h=16$			&$H/h=32$\\
 \hline 
 \\
Edge		&$8.29\mathrm{e}{5} ~~ (69)$		&$33.84 ~ (27)$		&$37.53  ~ (35)$		&$44.98 ~ (46)$	 \\
 \hline
Face 1-4 and Edge	& $3.90\mathrm{e}{5} ~ (167)$		&$25.72 ~ (31)$		&$34.96 ~ (39)$		&$46.95 ~ (49)$	\\
 \multicolumn{5}{c}{}\\
 \multicolumn{5}{c}{Wirebasket based coarse space}\\
 &No enrichment &\multicolumn{3}{c}{Face enrichment: $\lambda_{\mathcal{F}} < \lambda^*$} \\
 && $\lambda^*=0.075$	&$\lambda^*=0.035$			&$\lambda^*=0.0187$ \\
 $\alpha$ distribution &$H/h=8$		&$H/h=8$			&$H/h=16$			&$H/h=32$\\
 \hline 
 \\		
Edge		&$21.07 ~ (26)$		&$12.49 ~ (26)$			&$18.28 ~ (27)$			&$31.42 ~ (33)$	   \\ \hline
Face 1-4 and Edge	&$1.26\mathrm{e}{5} ~ (99)$		&$15.63 ~ (25)$			&$21.09 ~ (31)$			&$29.26 ~ (32)$
\\ \\
\end{tabular}
\end{center}
\caption{Condition number estimates and iteration counts (inside brackets) for the two preconditioners, the vertex based and the wirebasket based. The rows correspond to different $\alpha$ distribution as shown in Fig.~\ref{fig:fourfaces}--\ref{fig:twoedges}, and the columns correspond to no enrichment and with enrichment.}\label{tab:results_on_edges}
\end{table}

%
\begin{table}[htp]
\begin{center}
\begin{tabular}{ccccc}
			\multicolumn{1}{c}{Each vertical edge} 	& \multicolumn{2}{c}{Each face on the $xz$ plane} 					& \multicolumn{2}{c}{Each face on the $yz$ plane} 	 \\
\hline
$\lambda_{\mathcal{E}}$ 	& $\lambda_{\mathcal{F}_I}$ 	& $\lambda_{\mathcal{F}}$ 	& $\lambda_{\mathcal{F}_I}$ 	& $\lambda_{\mathcal{F}}$ 	 \\
  \hline	
$3.3\mathrm{e}{-8}$		& $0.0$				&$7.2\mathrm{e}{-3}$		& $0.0$				&$\mathbf{2.7\mathrm{e}{-2}}$	\\
$7.5\mathrm{e}{-8}$		& $4.5\mathrm{e}{-8}$		&$\mathbf{4.0\mathrm{e}{-2}}$ 	& $5.8\mathrm{e}{-7}$		&				\\
$\mathbf{6.8\mathrm{e}{-2}}$	& $8.2\mathrm{e}{-3}$		&				& $3.8\mathrm{e}{-3}$		&				\\
				& $1.8\mathrm{e}{-2}$		&				& $1.4\mathrm{e}{-2}$		&				\\
				& $3.1\mathrm{e}{-2}$		&				& $3.2\mathrm{e}{-2}$		&				\\
				& $3.1\mathrm{e}{-2}$		&				& $\mathbf{3.9\mathrm{e}{-2}}$	\\
				& $\mathbf{4.3\mathrm{e}{-2}}$	&				& 				\\				
\end{tabular}
\end{center}
\caption{Showing the eigenvalues for $H/h=32$ and the different eigenvalue problems, which (except for the one in boldface) are below the threshold and have been selected for the enrichment of the two coarse spaces, and the one (in boldface) which is just above the threshold, cf. Table \ref{tab:edge eigenvalues}. $\lambda_{\mathcal{E}}$ represent the eigenvalues of the edge eigenvalue problem, $\lambda_{\mathcal{F}_I}$ and $\lambda_{\mathcal{F}}$ represent the eigenvalues of the face eigenvalue problems for the vertex based and the wirebasket based coarse space, respectively.} 
\label{tab:edge eigenvalues}
\end{table}

We have also done experiments with varying jumps in the coefficients for each distribution shown, and found the methods robust with respect to the jumps, and inclusions cutting through faces and edges. Of the two, the wirebasket based coarse space has a convergence which is faster than the vertex based coarse space, however, the size of the coarse space is much larger in the former one.


%

\begin{thebibliography}{10}

\bibitem{Bjorstad-Koster-Krzyzanowski:HPC:2002}
{\sc Petter~E. Bj{\o}rstad, Jacko Koster, and Piotr Krzy\.zanowski}, {\em
  Domain decomposition solvers for large scale industrial finite element
  problems}, in Applied Parallel Computing. New Paradigms for HPC in Industry
  and Academia, vol.~1947 of Lecture Notes in Computer Science,
  Springer-Verlag, 2001, pp.~373--383.

\bibitem{Bjorstad-Krzyzanowski:PPAM2001}
{\sc Petter~E. Bj{\o}rstad and Piotr Krzy\.zanowski}, {\em Flexible 2-level
  {N}eumann-{N}eumann method for structural analysis problems}, in Proceedings
  of the 4th International Conference on Parallel Processing and Applied
  Mathematics, PPAM2001 Naleczow, Poland, September 9-12, 2001, vol.~2328 of
  Lecture Notes in Computer Science, Springer-Verlag, 2002, pp.~387--394.

\bibitem{Brenner:2008:MTF}
{\sc Susanne~C. Brenner and L.~Ridgway Scott}, {\em The Mathematical Theory of
  Finite Element Methods}, vol.~15 of Texts in Applied Mathematics, Springer
  Science \& Business Media, New York, third~ed., 2008.

\bibitem{calvo2015adaptive}
{\sc Juan~G. Calvo and Olof~B Widlund}, {\em An adaptive choice of primal
  constraints for {BDDC} domain decomposition algorithms}, Tech. Report
  TR2015-979, Courant Institute of Mathematical Sciences, 2016.

\bibitem{Chartier:2003:SAMGe}
{\sc T.~Chartier, R.~D. Falgout, V.~E. Henson, J.~Jones, T.~Manteuffel,
  S.~McCormick, J.~Ruge, and P.~S. Vassilevski}, {\em Spectral {AMG}e
  ({$\rho$}{AMG}e)}, SIAM J. Sci. Comput., 25 (2003), pp.~1--26.

\bibitem{chung2016adaptive}
{\sc Eric Chung, Yalchin Efendiev, and Thomas~Y Hou}, {\em Adaptive multiscale
  model reduction with generalized multiscale finite element methods}, Journal
  of Computational Physics, 320 (2016), pp.~69--95.

\bibitem{dolean:2015:IDD}
{\sc Victorita Dolean, Pierre Jolivet, and Fr{\'e}d{\'e}ric Nataf}, {\em An
  Introduction to Domain Decomposition Methods}, Society for Industrial and
  Applied Mathematics, Philadelphia, 2015.

\bibitem{Dolean:2012:ATL}
{\sc V.~Dolean, F.~Nataf, R.~Scheichl, and N.~Spillane}, {\em Analysis of a
  two-level schwarz method with coarse spaces based on local
  {D}irichlet-to-{N}eumann maps}, Comput. Methods Appl. Math., 12 (2012),
  pp.~391--414.

\bibitem{Efendiev:2012:RTD}
{\sc Yalchin Efendiev, Juan Galvis, Raytcho Lazarov, Svetozar Margenov, and Jun
  Ren}, {\em Robust two-level domain decomposition preconditioners for
  high-contrast anisotropic flows in multiscale media}, Comput. Methods Appl.
  Math., 12 (2012), pp.~415--436.

\bibitem{Efendiev:2012:RDD}
{\sc Yalchin Efendiev, Juan Galvis, Raytcho Lazarov, and J{\o}rg Willems}, {\em
  Robust domain decomposition preconditioners for abstract symmetric positive
  definite bilinear forms}, ESAIM Math. Mod. Num. Anal., 46 (2012),
  pp.~1175--1199.

\bibitem{Galvis:2010:DDM}
{\sc Juan Galvis and Yalchin Efendiev}, {\em Domain decomposition
  preconditioners for multiscale flows in high-contrast media}, Multiscale
  Model. Simul., 8 (2010), pp.~1461--1483.

\bibitem{Galvis:2010:DDMR}
\leavevmode\vrule height 2pt depth -1.6pt width 23pt, {\em Domain decomposition
  preconditioners for multiscale flows in high-contrast media: Reduced
  dimension coarse spaces}, Multiscale Model. Simul., 8 (2010), pp.~1621--1644.

\bibitem{atle2015harmonic}
{\sc Martin~J. Gander, Atle Loneland, and Talal Rahman}, {\em Analysis of a new
  harmonically enriched multiscale coarse space for domain decompostion
  methods}, Submitted,  (2015).

\bibitem{graham2007domain}
{\sc I.~G. Graham, P.~O. Lechner, and R.~Scheichl}, {\em Domain decomposition
  for multiscale {PDE}s}, Numerische Mathematik, 106 (2007), pp.~589--626.

\bibitem{kim2016bddc}
{\sc Hyea~Hyun Kim, Eric Chung, and Junxian Wang}, {\em {BDDC} and {FETI-DP}
  algorithms with adaptive coarse spaces for three-dimensional elliptic
  problems with oscillatory and high contrast coefficients}, arXiv:1606.07560,
  (2016).

\bibitem{kim2015bddc}
{\sc Hyea~Hyun Kim and Eric~T Chung}, {\em A {BDDC} algorithm with enriched
  coarse spaces for two-dimensional elliptic problems with oscillatory and high
  contrast coefficients}, Multiscale Modeling \& Simulation, 13 (2015),
  pp.~571--593.

\bibitem{klawonn:2016:adaptive}
{\sc Axel Klawonn, Martin Kühn, and Oliver Rheinbach}, {\em Adaptive coarse
  spaces for {FETI-DP} in three dimensions}, SIAM Journal on Scientific
  Computing, 38 (2016), pp.~A2880--A2911.

\bibitem{KRR:2015:FDMACS}
{\sc Axel Klawonn, Patrick Radtke, and Oliver Rheinbach}, {\em {FETI-DP}
  methods with an adaptive coarse space}, SIAM J. Numer. Anal., 53 (2015),
  pp.~297--320.

\bibitem{klawonn2016comparison}
\leavevmode\vrule height 2pt depth -1.6pt width 23pt, {\em A comparison of
  adaptive coarse spaces for iterative substructuring in two dimensions},
  Electronic Transactions on Numerical Analysis, 45 (2016), pp.~75--106.

\bibitem{Malek:2015:PCG}
{\sc Josef M{\'a}lek and Zden{\v{e}}k Strako{\v{s}}}, {\em Preconditioning and
  the conjugate gradient method in the context of solving {PDE}s}, vol.~1 of
  SIAM Spotlights, Society for Industrial and Applied Mathematics (SIAM),
  Philadelphia, PA, 2015.

\bibitem{mandel2007adaptive}
{\sc Jan Mandel and Bed{\v{r}}ich Soused{\'\i}k}, {\em Adaptive selection of
  face coarse degrees of freedom in the {BDDC} and the {FETI-DP} iterative
  substructuring methods}, Computer methods in applied mechanics and
  engineering, 196 (2007), pp.~1389--1399.

\bibitem{Mathew:2008:DDM}
{\sc Tarek P.~A. Mathew}, {\em Domain Decomposition Methods for the Numerical
  Solution of Partial Differential Equations}, vol.~61 of Lecture Notes in
  Computational Science and Engineering, Springer-Verlag, Berlin, Heilderberg,
  2008.

\bibitem{Nataf:2011:CSC}
{\sc Fr{\'e}d{\'e}ric Nataf, Hua Xiang, Victorita Dolean, and Nicole Spillane},
  {\em A coarse space construction based on local {D}irichlet-to-{N}eumann
  maps}, SIAM J. Sci. Comput., 33 (2011), pp.~1623--1642.

\bibitem{oh:2016:BDDC}
{\sc Duk-Soon Oh, Olof~B Widlund, Stefano Stefano~Zampini, and Clark~R.
  Dohrmann}, {\em {BDDC} algorithms with deluxe scaling and adaptive selection
  of primal constraints for {R}aviart--{T}homas vector fields}, Tech. Report
  TR2015-978, Courant Institute of Mathematical Sciences, 2015.

\bibitem{Pechstein2013:FBT}
{\sc Clemens Pechstein}, {\em Finite and Boundary Element Tearing and
  Interconnecting Solvers for Multiscale Problems}, vol.~90 of Lecture Notes in
  Computational Science and Engineering, Springer-Verlag, Berlin, Heilderberg,
  2013.

\bibitem{pechstein2016unified}
{\sc C~Pechstein and C.~R. Dohrmann}, {\em A unified framework for adaptive
  {BDDC}}, Tech. Report RICAM-Report 2016-20, Johann Radon Institute for
  Computational and Applied Mathematics, 2016.

\bibitem{Pechstein:2013:WPI}
{\sc Clemens Pechstein and Robert Scheichl}, {\em Weighted poincar\'e
  inequalities}, IMA J Numer. Anal., 32 (2013), pp.~652--686.

\bibitem{smith2004domain}
{\sc Barry Smith, Petter Bjorstad, and William Gropp}, {\em Domain
  Decomposition: Parallel Multilevel Methods for Elliptic Partial Differential
  Equations}, Cambridge University Press, 2004.

\bibitem{sousedik2013adaptive}
{\sc Bed{\v{r}}ich Soused{\'\i}k, Jakub {\v{S}}{\'\i}stek, and Jan Mandel},
  {\em Adaptive-multilevel {BDDC} and its parallel implementation}, Computing,
  95 (2013), pp.~1087--1119.

\bibitem{Spilane:2014:ARC}
{\sc N.~Spillane, V.~Dolean, P.~Hauret, F.~Nataf, C.~Pechstein, and
  R.~Scheichl}, {\em Abstract robust coarse spaces for systems of {PDE}s via
  generalized eigenproblems in the overlaps}, Numer. Math., 126 (2014),
  pp.~741--770.

\bibitem{toselli2005domain}
{\sc Andrea Toselli and Olof~B. Widlund}, {\em Domain Decomposition Methods:
  Algorithms and Theory}, vol.~34, Springer, 2005.

\end{thebibliography}

\end{document}